\pdfoutput=1
\documentclass{amsart}

\title[Factorizations of surface diffeomorphisms]{Factorizations of diffeomorphisms of compact surfaces with boundary}
\author{Andy Wand}
\date{}

\usepackage{amsmath, MnSymbol, amsthm, graphicx, verbatim, hyperref}

\newcommand{\pg}{\Sigma}
\newcommand{\mon}{\varphi}

\newcommand{\mcg}{\Gamma_\pg}
\newcommand{\veer}{Veer(\pg)}

\newcommand{\obd}{(\Sigma,\varphi)}

\newcommand{\prey}{y^\alpha}

\newcommand{\preB}{B^\alpha}

\newcommand{\dpoint}{\textbullet -point }
\newcommand{\cpoint}{{\large \textopenbullet}-point }
\newcommand{\dpoints}{\textbullet -points }
\newcommand{\cpoints}{{\large \textopenbullet}-points }

\newcommand{\refy}{\bar{y}}

\newcommand{\refB}{\bar{B}}

\newcommand{\refpg}{\bar{\Sigma}}
\newcommand{\refmon}{\bar{\mon}}

\newcommand{\pos}{\mathcal{P}}

\newcommand{\G}{\mathcal{G}}
\newcommand{\B}{\mathcal{B}}
\newcommand{\R}{\mathcal{R}}

\newcommand{\aclass}{{[\alpha]}}
\newcommand{\w}{\widetilde}

\newcommand{\fac}{\tau_{\alpha_n} \cdots \tau_{\alpha_1}}
\newcommand{\reffac}{\tau^{-1}_{\alpha_n} \cdots \tau^{-1}_{\alpha_1}}

\newcommand{\aequiv}{\stackrel{\alpha}{\sim}}
\newcommand{\bequiv}{\stackrel{b}{\sim}}

\newtheorem{thm}{Theorem}[section]
\newtheorem{cor}[thm]{Corollary}
\newtheorem{lem}[thm]{Lemma}

\theoremstyle{definition}
\newtheorem{defin}[thm]{Definition}

\theoremstyle{remark}
\newtheorem{remark}[thm]{Remark}
\newtheorem{example}[thm]{Example}
\newtheorem{obs}[thm]{Observation}

\begin{document}

\begin{abstract}

We study diffeomorphisms of compact, oriented surfaces, developing
methods of distinguishing those which have positive factorizations
into Dehn twists from those which satisfy the weaker condition of
right veering. We use these to construct open book decompositions of
Stein-fillable 3-manifolds whose monodromies have no positive
factorization.

\end{abstract}

\maketitle

\section{Introduction}

Let $\pg$ be a compact, orientable surface with
nonempty boundary. The (restricted) mapping class group of $\pg$,
denoted $\mcg$, is the group of isotopy classes of orientation
preserving diffeomorphisms of $\pg$ which restrict to the identity
on $\partial\pg$. The goal of this paper is to study the monoid $Dehn^+(\pg) \subset \mcg$ of products of right Dehn twists, and apply this study to the question of the extent to which properties of a contact three-manifold are captured by the data associated to an arbitrary supporting open book decomposition.

In particular, we introduce the idea of a \emph{right position} for the images of a collection of disjoint properly embedded arcs in $\pg$ under a given diffeomorphism, and develop consistency conditions which allow us to prove:

\begin{thm}\label{thm_rp}
Let $\pg$ be a compact surface with boundary, $\mon \in Dehn^+(\pg)$, and $\{\gamma_i\}_{i=1}^n$ a collection of disjoint, properly embedded arcs in $\pg$. Then $\{(\mon,\gamma_i)\}_{i=1}^n$ admit consistent right positions.
\end{thm}

This result can be thought of as a refinement of the \emph{right veering} condition, which dates at least as far back as Thurston's proof of the left orderability of the braid group, and which was introduced into the study of contact structures by Honda, Kazez, and Matic in \cite{hkm}.  In that paper, it is shown that the monoid $\veer \subset \mcg$ of right-veering diffeomorphisms (see Section~\ref{sec_p} for definitions) strictly contains $Dehn^+(\pg)$. Our methods allow one to easily distinguish certain elements of $\veer \setminus Dehn^+(\pg)$.

 We further use Theorem~\ref{thm_rp} to develop necessary conditions for elements of the set of curves which can appear as twists in some positive factorization of a given $\mon \in \mcg$. 
 
\begin{thm}\label{thm_nested}

Let $\gamma_1$ and $\gamma_2$ be disjoint properly embedded arcs in a surface $\pg$, $\mon \in \mcg$, and $\mon(\gamma_1)$ and $\mon(\gamma_2)$ flat with respect to some $D$. Then $\tau_\alpha$ is a Dehn twist in some positive factorization of $\mon$ only if $\alpha$ is nested with respect to $D$.

\end{thm} 
 
 (The non-standard terminology of Theorem~\ref{thm_nested} will be defined in Section~\ref{sec_d}.)

 Finally, as an application of the above methods, we have
 
\begin{thm}\label{thm_stein_but_not_positive}
There exist open book decompositions which support Stein fillable contact structures but whose monodromies cannot be factorized into positive Dehn twists.
\end{thm}

Central to much of current research in contact geometry is the relation between the
monodromy of an open book decomposition and geometric properties of the contact structure, such as tightness and fillability. The starting point is the remarkable theorem of Giroux \cite{gi}, demonstrating a one-to-one correspondence between oriented
contact structures (up to isotopy) on a 3-manifold $M$ and open book decompositions
of $M$ up to positive stabilization. It has been shown by Giroux \cite{gi}, Loi and Peirgallini \cite{lp}, and Akbulut and Ozbagci \cite{ao}, that any open book
with monodromy which can be factorized into positive Dehn twists
supports a Stein-fillable contact structure, that every Stein-fillable contact structure is supported by \emph{some} open book with monodromy which can be factored into positive twists, and in \cite{hkm} that a contact structure is tight if and only if the monodromy of each supporting open book is
right-veering. The question of whether \emph{each} open book which supports a Stein-fillable contact structure must be positive is then answered in the negative by our Theorem~\ref{thm_stein_but_not_positive}.

Section~\ref{sec_p} introduces some conventions and definitions, and
provides motivation for what is to come. Section~\ref{sec_rp}
is devoted to the terminology and proof of Theorem~\ref{thm_rp}. We also construct simple examples of open books whose monodromies, though right veering, have no positive factorization.

In Section~\ref{sec_d} we use these results to address the question of under what conditions various isotopy classes of curves on a surface can be shown not to appear as Dehn twists in any positive factorization of a given monodromy. We obtain various necessary conditions under certain assumptions on the monodromy, culminating in Theorem~\ref{thm_nested}.

Finally, in Section~\ref{sec_e}, we construct explicit examples of open book decompositions for an infinite family of 3-manifolds which support Stein fillable contact structures yet whose monodromies have no positive factorizations. For the construction we demonstrate a method of modifying a certain mapping class group relation (the lantern relation) into an `immersed' configuration, which we then use to modify certain open books with positively factored monodromies (which therefore support Stein fillable contact manifolds) into stabilization-equivalent open books (which support the same contact structures) whose monodromies now have non-trivial negative twisting. We then apply the methods developed in the previous sections to show that in fact this negative twisting is essential.

As this paper was being written, Baker, Etnyre, and Van Horn-Morris
\cite{bev} were able to construct similar examples of
non-positive open books supporting Stein fillable contact structures, all of which use the same surface,
$\pg_{2,1}$, involved in our construction. Their method for demonstrating non-positivity is on the one hand substantially shorter, but on the other is quite restrictive, being entirely specific to this surface, and also requiring complete knowledge of the Stein fillings of the contact manifold in question, which is known only for a small class of contact manifolds.

We would like to thank Danny Calegari for many helpful comments on a previous version of this paper, Rob Kirby for his support and encouragement throughout the much extended period over which it was completed, and the referee for many helpful suggestions and comments. We would also like to acknowledge the support and hospitality of the Max Planck Institute for Mathematics.

\section{Preliminaries}\label{sec_p}

Throughout, $\pg$ denotes a compact, orientable surface with
nonempty boundary. The (restricted) mapping class group of $\pg$,
denoted $\mcg$, is the group of isotopy classes of orientation
preserving diffeomorphisms of $\pg$ which restrict to the identity
on $\partial \pg$. In general we will not distinguish between a
diffeomorphism and its isotopy class.

Let $SCC(\pg)$ be the set of simple closed curves
on $\pg$. Given $\alpha \in SCC(\pg)$ we may define a self-diffeomorphism $D_\alpha$ of $\pg$ which is supported near $\alpha$ as follows. Let a neighborhood $N$ of $\alpha$ be identified by oriented coordinate charts with the annulus $\{a \in \mathbb{C} \ | \ 1 \leq ||a|| \leq 2 \}$. Then $D_\alpha$ is the map $a \mapsto e^{-i2\pi (||a|| -1)}a$ on $N$, and the identity on $\pg \setminus N$. We call $D_\alpha$ the \emph{positive Dehn
twist} about $\alpha$. The inverse
operation, $D^{-1}_\alpha$, is a \emph{negative} Dehn twist. We denote the isotopy class of $D_\alpha$ by $\tau_\alpha$. It
can easily be seen that $\tau_\alpha$ depends
only on the isotopy class of $\alpha$.

We call $\mon \in \mcg$ \emph{positive} if it can be factored as a
product of positive Dehn twists, and denote the monoid of such mapping classes as $Dehn^+(\pg)$. 

An open book decomposition $\obd$, where $\mon \in \mcg$, for a
3-manifold $M$, with binding $K$, is a homeomorphism between $((\pg
\times[0,1] )/ \sim_\mon, (\partial\pg \times[0,1] )/ \sim_\mon)$
and $(M,K)$. The equivalence relation $\sim_\mon$ is generated by
$(x, 0) \sim_\mon (\mon(x), 1)$ for $x \in \pg$ and $(y, t)
\sim_\mon (y, t')$ for $y \in \partial\pg$.

\begin{defin}\label{def_pe}
 For a surface $\pg$ and positive mapping class $\varphi$, we define the \emph{positive extension} $p.e.(\varphi) \subset SCC(\pg)$ as the set of all $\alpha \in SCC(\pg)$ such that $\tau_\alpha$ appears in some positive factorization of $\varphi$.
\end{defin}

A recurring theme of this paper is to use properly embedded arcs
$\gamma_i \hookrightarrow \pg$ to understand restrictions on
$p.e.(\mon)$ which can be derived from the
monodromy images $\mon(\gamma_i)$ in $\pg$ (relative to the arcs themselves).
Our general method is as follows: Suppose that $P$ is some property
of pairs $(\mon(\gamma),\gamma)$ (which we abbreviate by referring
to $P$ as a property of the image $\mon(\gamma)$) which holds for
the case $\mon$ is the identity, and is preserved by positive Dehn
twists. Suppose then that $\alpha \in p.e.(\mon)$; then there is a positive factorization
of $\mon$ in which $\tau_{\alpha}$ is a twist. Using the well known
braid relation, we can assume $\tau_{\alpha}$ is the \emph{final}
twist, and so the monodromy given by $\tau^{-1}_{\alpha} \circ \mon$
is also positive. It follows that $P$ holds for $(\tau^{-1}_{\alpha} \circ \mon)(\gamma)$
as well.

\begin{example}\label{example_arcs}
As a motivating example, consider a pair of arcs $\gamma,\gamma':[0,1]
\hookrightarrow \pg$ which share an endpoint $\gamma(0) = \gamma'(0)
= x \in \partial \pg$, isotoped to minimize intersection. Following
\cite{hkm}, we say $\gamma'$ is `to the right' of $\gamma$, denoted
$\gamma' \geq \gamma$, if either the pair is isotopic, or if the
tangent vectors $(\dot{\gamma'}(0),\dot{\gamma}(0))$ define the
orientation of $\pg$ at $x$. The property of being `to the right' of
$\gamma$ (at $x$) is then a property of images $\mon(\gamma)$ which
satisfies the conditions of the previous paragraph. We conclude that, if $\mon(\gamma)$ is to the right of $\gamma$, then
$\alpha \in p.e.(\mon)$ only if $(\tau^{-1}_{\alpha} \circ
\mon)(\gamma)$ is to the right of $\gamma$.
\end{example}

The main motivation for the use of properly embedded arcs to characterize properties of mapping classes is the well-known result that a mapping class of a surface with boundary is completely determined by the images of a set of properly embedded disjoint arcs whose complement is a disc (an application of the `Alexander method', see e.g.~\cite{fm}). The most well-known example of such a property is what is now generally known as \emph{right veering}:

\begin{defin}\label{def_right_veering}\cite{hkm}
Let $\mon$ be a mapping class in $\mcg$, $\gamma \hookrightarrow \pg$ a properly embedded arc with endpoint $x \in \partial \pg$. Then $\mon$ is \emph{right veering} if, for each such $\gamma$ and $x$, the image $\mon(\gamma)$ is to the right of $\gamma$ at $x$.
\end{defin}

We denote the set of isotopy classes of right veering diffeomorphisms as $\veer \subset \mcg$.

The right veering property can perhaps best be thought of as a necessary condition of positivity, and (by~\cite{hkm}), using the Giroux correspondence, also a necessary condition for tightness of the supported contact structure. It was however shown in that paper that right veering is far from a sufficient condition; indeed, that \emph{any} open book decomposition may be stabilized to a right veering one. 

The failure of right veering to capture either tightness or positivity is at least partially due to two obvious shortcomings. Firstly, we observe that right veering `only sees one arc at a time'. Referring back to the original motivation of the Alexander method, we would rather a property which can see arbitrary collections of (disjoint) arcs. Secondly, right veering is `localized' to the boundary; in particular negative twisting in the interior of a surface may be hidden by positivity nearer to the boundary. The goal of Section~\ref{sec_rp} is to introduce a substantial refinement of right veering which takes into account each of these issues.

\section{Right position}\label{sec_rp}

This section is devoted to the proof of Theorem \ref{thm_rp}. We will begin with an introduction to the rather non-standard terminology, and then turn to the various ideas involved in the proof.

\subsection{Definitions and examples}

\begin{defin}\label{rp}
Suppose  $\gamma \hookrightarrow \pg$ is a properly embedded arc with
boundary $ \{c,c'\}$, and $\mon \in \veer$, with
$\mon(\gamma)$ isotoped to minimize $\gamma \cap \mon(\gamma)$. Let $A$ be a subset of the positively oriented interior intersections in $\gamma \cap \mon(\gamma)$ (throughout the paper, each properly embedded arc $\gamma$ and its image $\mon(\gamma)$ will be given opposite orientations, while a point $p \in \gamma \cap \mon(\gamma)$ will be considered \emph{positive} if the tangent vectors of $\gamma$ and $\mon(\gamma)$ in that order define the orientation of $\pg$ at $p$, negative otherwise - sign conventions are illustrated in Figure \ref{fig_signs}). Then the set $\{c,c'\} \cup A$ is a \emph{right position} $\pos = \pos(\mon,\gamma)$ for the pair $(\mon,\gamma)$.  If $\{\gamma_i\}_{i=1}^n$ is a collection of disjoint properly embedded arcs, and $\pos_i$ a right position for each pair $(\mon,\gamma_i)$, then we refer to  $\bigcup_{i=1}^n  \pos_i$ as a right position for the pair $(\mon,\{\gamma_i\}_{i=1}^n)$.
\end{defin}

Note that we are thinking of the set $I$ of positively oriented interior intersections in $\gamma \cap \mon(\gamma)$ as depending only on the data of $\gamma$ and the isotopy class $\mon$, and thus independent of isotopy of $\gamma$ and $\mon(\gamma)$ as long as intersection minimality holds. In particular, the integer $|I|$ clearly depends only on this data. We may label elements of $I$ with indices $1,\ldots, |I|$ increasing along $\gamma$ from distinguished endpoint $c$. Right positions for the pair $(\mon, \gamma)$ are thus in 1-1 correspondence with subsets of the set of these indices.

\begin{figure}[h]
\centering \scalebox{1.5}{\includegraphics{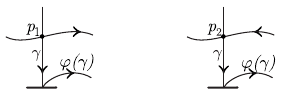}} \caption[Sign conventions]{$p_1$
is a positive, $p_2$ a negative, point of $\gamma \cap
\mon(\gamma)$} \label{fig_signs}
\end{figure}

Associated to a right position is the set $H(\pos) := \{[v,v']
\subset \mon(\gamma) \ | \ v,v' \in \pos\}$. We denote these
segments by $h_{v,v'}$, or, if only a single endpoint is required,
simply use $h_v$ to denote a subarc starting from $v$ and
extending as long as the context requires. In this case direction
of extension along $\mon(\gamma)$ will be clear from context.

Right position can thus be thought of as a way of decomposing the image $\mon(\gamma)$ into `horizontal' segments $H(\pos)$, separated by the points
$\pos$ (Figure~\ref{fig_rp_example_1}). Note that for any right-veering $\mon$, and any properly embedded arc $\gamma$, $(\mon, \gamma)$
has a trivial right position consisting of the points $\partial(\gamma)$, and so there is a single
horizontal segment $h_{c,c'} = \mon(\gamma)$. Of course if $I$ is nonempty, there are $2^{|I|} - 1$
non-trivial right positions. Intuitively, the points in $\pos$ play a
role analogous to the boundary points of the arc, allowing
us to localize the global `right'ness of an arc image though the
decomposition into horizontal segments.

\begin{figure}[h!]
\centering \scalebox{.9}{\includegraphics{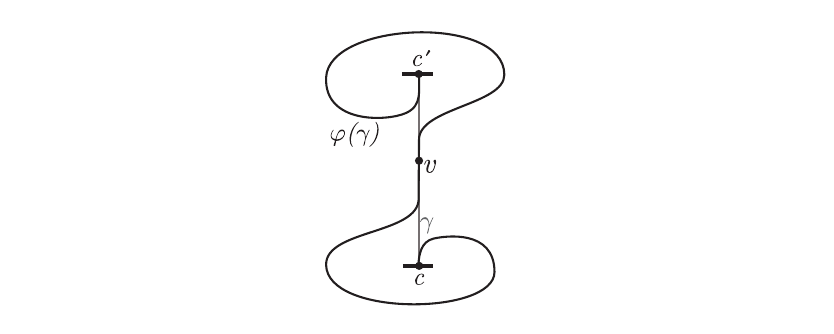}}
\caption[Right position example]{$\pos = \{c,v,c'\}$ is a right position of $(\mon, \gamma)$. There are three distinct horizontal segments: $h_{c,v}, h_{v,c'}$, and $h_{c,c'}$}
\label{fig_rp_example_1}
\end{figure}

We are interested in using right position to compare \emph{pairs} of
arcs; the idea being that, if $\mon$ is positive, we expect to be
able to view any pair of horizontal segments as either unrelated, or
as belonging to the `same' horizontal segment.

To get started, we need to be able to compare horizontal segments.
To that end, we have:

\begin{defin}\label{def_rectangular_region} Suppose $\{\gamma_i\}_{i=1}^n$ are properly embedded arcs in a surface $\pg$, and $\mon \in \veer$. A \emph{rectangular region} of $(\pg,\mon, \{\gamma_i\})$ is the image of an immersion of the standard disc $[0,1] \times [0,1]$, such that each edge is embedded into one of the $\gamma_i$, and vertices map to intersections $\gamma_i \cap \gamma_j$.
\end{defin}

\begin{defin}\label{def_initially_parallel}
Let $\{\gamma_i\}_{i=1}^n$ be a collection of disjoint, properly embedded arcs in $\pg$, $\mon \in \veer$, and $\pos_i = \pos_i(\mon,\gamma_i)$ a right position for each $i$. Associated to this data is then a collection $\R(\pos_1,\ldots,\pos_n)$ of rectangular regions of $(\pg,\mon, \{\gamma_i\} \cup \{\mon(\gamma_i)\})$ such that for each region (1) there exist orientations of the $\gamma_i$ which induce an orientation of the boundary of the region, and (2) exactly two vertices lie in $\cup\pos_i$ (see Figure \ref{fig_regions}). We refer to this collection as the \emph{initially parallel} regions associated to $\{\pos_i\}$.
\end{defin}

 Given an initially parallel region, we refer to the pair of vertices in $\cup\pos_i$ as \emph{\dpoints}, \emph{\cpoints} otherwise (so \cpoints and \dpoints alternate along the boundary of a region). Throughout the paper, we will draw \dpoints as solid dots, \cpoints as circles.

\begin{defin}\label{def_orientations}
Let $B_1,B_2 \in \R(\pos_1,\pos_2)$ (as in the previous definition), and orient $\gamma_1$ and $\gamma_2$ such that the induced orientation of $\partial(B_1)$ agrees with the standard (counterclockwise) orientation of the boundary of a region in an oriented surface. If these orientations then induce the standard orientation of $\partial(B_2)$, we say that the pair have the same orientation, while if the orientations induce the opposite orientation on $\partial(B_2)$, we say the regions have the opposite orientations.
\end{defin}

\begin{figure}[h!]\label{fig_regions}
\centering \scalebox{2}{\includegraphics{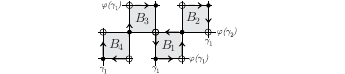}}
\caption{Various initially parallel regions in a collection $\R(\pos_1,\pos_2)$, where we have given the arcs the orientations which induce the standard orientation on $B_1$. Region $B_1$ has the opposite orientation of $B_2$ and $B_3$ (so $B_2$ and $B_3$ have the same orientation), while the given orientations do not orient $B_4$.}
\label{fig_regions}
\end{figure}

\begin{remark}
It is helpful to think of a rectangular region as the image under the covering map $\rho :\widetilde{\pg} \rightarrow \pg$ of an \emph{embedded} disc in the universal cover. As such, we will draw regions as embedded discs whenever doing so does not result in any loss of essential information.
\end{remark}

\begin{defin}\label{def_consistent}
Let $\{\gamma_i\}_{i=1}^n, n \geq 1,$ be a collection of
non-intersecting, properly embedded arcs in $\pg$, $\mon \in \veer$, and $\pos_i = \pos_i(\mon,\gamma_i)$ a right position
for each $i$. We say the $\pos_i$ are
\emph{consistent} if for each $B \in \R(\pos_1,\ldots,\pos_n)$, there is $B' \in \R(\pos_1,\ldots,\pos_n)$ with the same \cpoints, and opposite orientation, as $B$.  We refer to $B'$ as a \emph{completing region} for $B$. 
\end{defin}

\begin{example}\label{ex_rp}
The pair $(\mon,\gamma_1)$, $(\mon,\gamma_2)$ in Figure
\ref{fig_rp_example}(a), with the right positions indicated, is consistent, with two completed pairs of regions, as indicated in (b). It is straightforward to verify that there are no other regions in $\R(\pos_1,\pos_2)$.

\begin{figure}[h!]
\centering \scalebox{.8}{\includegraphics{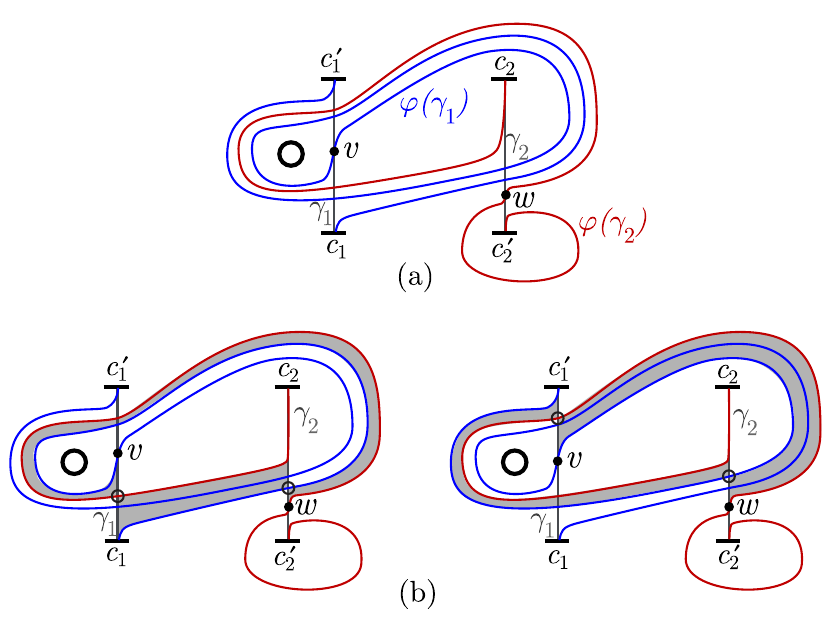}}
\caption[Right position examples]{(a): The pair $(\mon,\gamma_1)$, $(\mon,\gamma_2)$ admits consistent right positions. The pairs of completing discs are shaded in (b).} \label{fig_rp_example}
\end{figure}

The pair in Figure \ref{fig_rp_example_not}, however, admits no consistent right positions: Suppose
$\pos_1$ and $\pos_2$ are consistent right positions. Then $h_{c_1}$ and $h_{c_2}$ are edges of an initially parallel region $B$, so there must be $v \in \pos_1$ and $w \in
\pos_2$ satisfying the compatibility conditions. The only candidates are $c'_1$ and $c'_2$, and these do
not give a completing region. By Theorem
\ref{thm_rp}, this is sufficient to conclude that the mapping class has no
positive factorization.

\begin{figure}[h!]
\centering \scalebox{.5}{\includegraphics{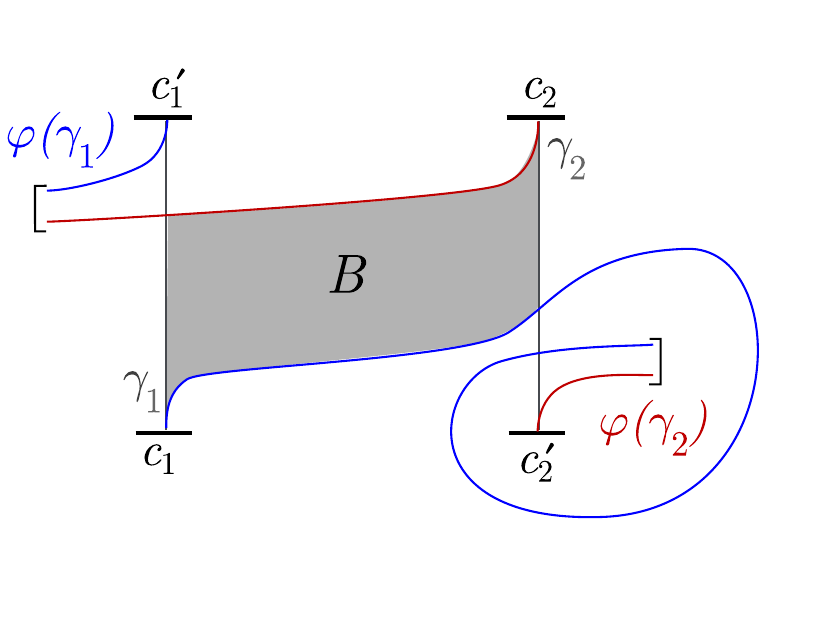}}
\caption[Not a right pair]{The pair $(\mon,\gamma_1)$, $(\mon,\gamma_2)$ does \emph{not} admit consistent right positions.}
\label{fig_rp_example_not}
\end{figure}

Finally, consider the open book decomposition indicated in Figure \ref{fig_rp_not_pos}, where the monodromy $\mon$ is $\tau^{-2}_\alpha \tau^{n_1}_{\beta_1} \tau^{n_2}_{\beta_2} \tau_{\beta_3} \tau_{\beta_4}$ (we have drawn the case $n_1 = 1$). Using the shaded initially parallel region, observe that the pair $(\mon,\gamma_1)$ and $(\mon, \gamma_2)$ admit no consistent right positions; indeed, for any $n_1, n_2 >0$ this will clearly remain true, as the picture will only be modified by twists about the boundary component parallel to $\beta_2$. In particular, $\mon$ has no positive factorization for any $n_1$ and $n_2$. On the other hand, it is a simple exercise to show that for any positive $n_1$ and $n_2$ each boundary component is `protected'; i.e. each properly embedded arc with endpoint on the boundary component under consideration maps to the right at that point (it suffices to check for either of the boundary components parallel to $\beta_3$ or $\beta_4$). In particular, each such monodromy is right-veering.  See e.g. \cite{l} for a more in-depth examination of this example.

\begin{figure}[htb]
\centering
\includegraphics[scale=.7]{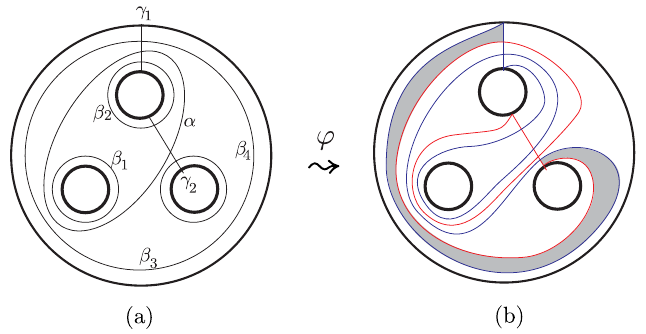}
\caption[Right veering but not positive]{(a) Surface and curves for the open book decomposition $(\pg_{0,4},\tau^{-2}_\alpha \tau^{n_1}_{\beta_1} \tau^{n_2}_{\beta_2} \tau_{\beta_3} \tau_{\beta_4})$. (b) The pair $(\mon,\gamma_1)$ (in blue), and $(\mon,\gamma_2)$ (in red) is \emph{not} a right pair, but $\mon$ is right veering for all $m,n \geq 1$.}
\label{fig_rp_not_pos}
\end{figure}

\end{example}

\subsection{The right position associated to a factorization}

Suppose we have a surface $\pg$, a mapping class $\mon \in \mcg$
given as a factorization $\omega$ of positive Dehn twists, and a set
$\{\gamma_i\}_{i=1}^n$ of pairwise non-intersecting, properly
embedded arcs in $\pg$. While each such arc admits at least one, and possibly many, right positions, the goal of this subsection is to give an algorithm which associates to $\omega$ a unique right position, denoted $\pos_\omega(\gamma_i)$ for each $i$, such that the set $\{\pos_\omega(\gamma_i)\}_{i=1}^n$ is consistent (Definition \ref{def_consistent}). Note that existence of such an algorithm proves Theorem \ref{thm_rp}.

We begin with the case of a single arc $\gamma$,
and construct $\pos_\omega(\gamma)$ by induction on the
number of twists $m$ in $\omega$.

\textbf{Base step: $\mathbf{m=1}$} Suppose $\mon = \tau_\alpha$, for
some $\alpha \in SCC(\pg)$. We begin by isotoping $\alpha$ so as to
minimize $\alpha \cap \gamma$. Label $\alpha \cap \gamma =
\{x_1, \ldots, x_p\}$, with indices increasing along $\gamma$. We
then define the right position $\pos_{\tau_\alpha}(\gamma)$ associated to
the twist $\tau_\alpha$ to be the points
$\{c=v_0,v_1,\ldots,v_{p-1}, c'=v_p\}$, where $\{c,c'\} =
\partial(\gamma)$, and, for $0 < i < p$, $v_i$ lies in the connected component
of $\gamma \setminus support(D_\alpha)$ between $x_i$ and
$x_{i+1}$ (Figure \ref{fig_base_step}).

\begin{figure}[h!]
\centering \scalebox{.5}{\includegraphics{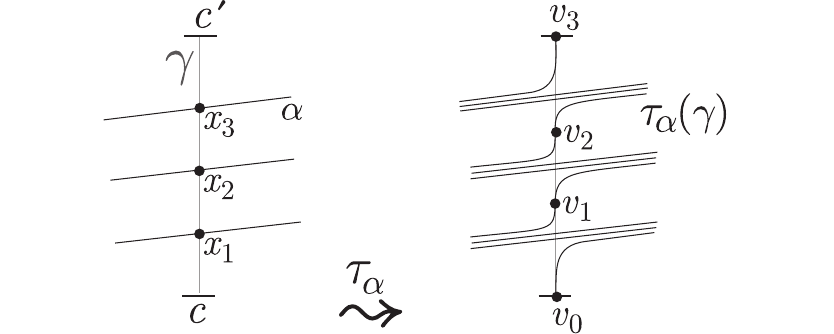}}
\caption[Algorithm - base case]{The base case: the right position $\pos_{\tau_\alpha}(\gamma)$ associated to a single Dehn twist.}
\label{fig_base_step}
\end{figure}

Now, this construction has as input only the isotopy class of $\alpha$, so we have:

\begin{lem}\label{lem_isotopy_independence_1}

The right position $\pos_{\tau_\alpha}(\gamma)$ defined in \textbf{base step} depends only on the isotopy class of $\alpha$.

\end{lem}

That for any collection of disjoint properly embedded arcs the resulting right positions are consistent will follow easily from observations to be made further on in this section. In the meantime, the impatient reader might enjoy (and the intuition of the patient reader might benefit from) some reflection on Figure \ref{fig_base_step_consistent} (b).

\begin{figure}[h!]
\centering \scalebox{1.3}{\includegraphics{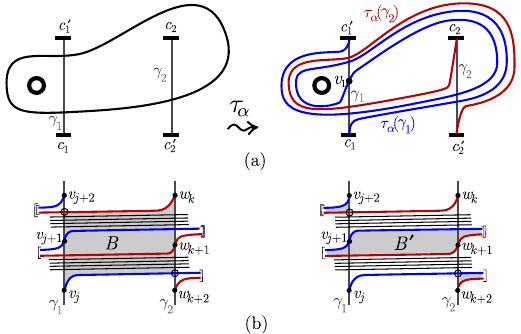}}
\caption[The base step]{(a) An example of the right position associated to a single Dehn twist. (b) A typical initially parallel region $B$ in the pair $\pos_{\tau_\alpha}(\gamma_1)$ and $\pos_{\tau_\alpha}(\gamma_2)$. Segments $h_{v_j}$ and $h_{w_k}$ are edges of the shaded initially parallel region in the leftmost figure, and are completed to $v_{j+2}, w_{k+2}$ by the shaded region in the rightmost figure}
\label{fig_base_step_consistent}
\end{figure}

\subsubsection{Description of the inductive step}

For the inductive step, we demonstrate an algorithm which has as input a surface $\pg$, mapping class $\mon \in \veer$, properly embedded arc $\gamma \hookrightarrow \pg$, right position $\pos(\mon,\gamma)$, and $\alpha \in SCC(\pg)$. The output then will be a unique right position $\pos^\alpha = \pos^\alpha(\tau_\alpha \circ \mon, \gamma)$. A key attribute of the algorithm will be that, while $\alpha$ is given as a particular representative of its isotopy class $[\alpha]$, the resulting right position $\pos^\alpha$ will be independent of choice of representative.

To keep track of things in an isotopy-independent way, we consider \emph{triangular regions} in $\pg$:

\begin{defin}
Suppose $\gamma$ and  $\gamma'$ are properly embedded arcs in a surface $\pg$, isotoped to minimize intersection. Then for $\alpha \in SCC(\pg)$, a \emph{triangular region} $T$ (of the triple $(\alpha,\gamma,\gamma')$) is the image of an immersion $f:\Delta \looparrowright \pg$, where $\Delta$ is a 2-simplex and the image of each edge is contained in one of $\alpha,\gamma$ or $\gamma'$.
\end{defin}

\begin{defin}\label{def_triangles}
 A triangular region $T$ for the ordered triple $(\alpha,\gamma,\gamma')$ is
 \begin{itemize}
   \item \emph{essential} if $\alpha$ can be isotoped relative to $\partial (T) \cap \alpha$ so as to intersect $\gamma$ and $\gamma'$ in a minimal number of points.
   \item \emph{upward} (\emph{downward}) if bounded by $\alpha,\gamma$ and $\gamma'$ in clockwise (counterclockwise) order.
 \end{itemize}
\end{defin}

\begin{figure}[h!]
\centering \scalebox{.9}{\includegraphics{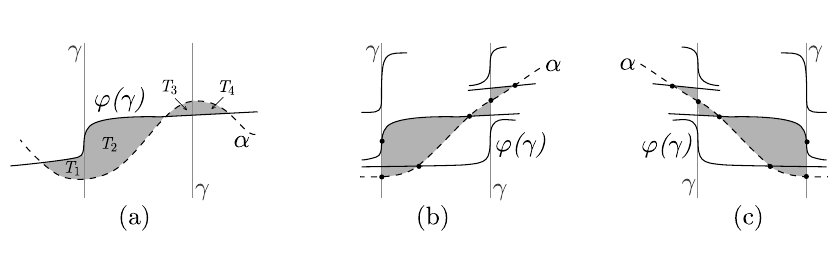}}
\caption[Regions and vertices]{(a) If the shaded region is a maximal bigon chain, then $T_1$ and $T_4$ are essential triangular regions for $(\alpha,\gamma,\mon(\gamma))$, $T_2$ and $T_3$ are not. (b) Upward triangular regions. (c) Downward triangular regions.}
\label{fig_regions_and_vertices}
\end{figure}

We begin with a brief description of the algorithm, with an illustrative example in Figure~\ref{fig_algorithm}. Let $\alpha$ be a representative of $[\alpha]$ which minimizes $\alpha \cap \gamma$ and $\alpha \cap \mon(\gamma)$, so in particular all triangular regions for the triple $(\alpha,\gamma,\mon(\gamma))$ are essential. Furthermore, chose $support(D_\alpha)$ so as not to intersect any point of $\gamma \cap \mon(\gamma)$ (Figure~\ref{fig_algorithm}(a)) (recall that $D_\alpha$ refers to the Dehn twist about fixed $\alpha \in [\alpha]$, while $\tau_\alpha$ is its isotopy class). Consider then the image $D_\alpha(\mon(\gamma))$, which differs from $\mon(\gamma)$ only in $support(D_\alpha)$ (Figure~\ref{fig_algorithm}(b)). Now, the only bigons bounded by the arcs $D_\alpha(\mon(\gamma))$ and $\gamma$ contain vertices which were in upward triangles in the original configuration. In particular, there is an isotopy of $D_\alpha(\mon(\gamma))$ over (possibly some subset of) these bigons which minimizes $\gamma \cap \tau_\alpha(\mon(\gamma))$ (the issue of exactly when a proper subset of the bigons suffices is taken up below, in Section \ref{details}). There is thus an inclusion of those points of $\pos$ which are \emph{not} in upward triangles into the set of positive intersections of $\gamma \cap \tau_\alpha(\mon(\gamma))$, whose image is therefore a right position (Figure~\ref{fig_algorithm}(c)).

\begin{figure}[h!]
\centering \scalebox{.9}{\includegraphics{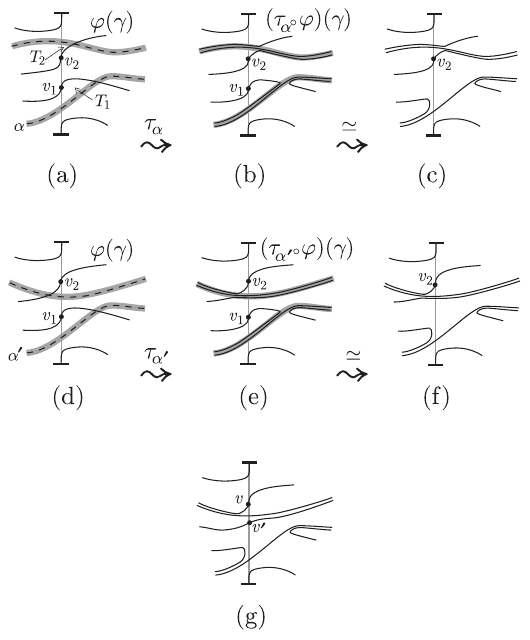}}
\caption[Description of the algorithm]{Top row: (a) The setup for the above discussion. Triangular region $T_1$ is upward, while $T_2$ is downward. The support of the twist $D_\alpha$ is shaded. (b) The result of $D_\alpha$. There is a single bigon, in which $v_1$ is a vertex. (c) The result of isotoping $D_\alpha(\mon(\gamma))$ over this bigon so as to minimize intersection with $\gamma$. Bottom row: Same as the top, but with $\alpha'$. Note that while the images $D_\alpha(\mon(\gamma))$ and $D_{\alpha'}(\mon(\gamma))$ are of course isotopic, the right positions of (c) and (f) differ. Finally, (g) indicates the isotopy-independent right position (i.e. the points $v$ and $v'$, along with $\partial(\gamma)$) which our algorithm is meant to pick out.}
\label{fig_algorithm}
\end{figure}

Note however that, as described, this resulting right position is $\emph{not}$ independent of the choice of $\alpha$; indeed, if $T$ is any downward triangle, we may isotope the given $\alpha$ over $T$ (e.g. go from Figure~\ref{fig_algorithm}(a) to Figure~\ref{fig_algorithm}(d), with $T=T_2$) to obtain $\alpha'\in [\alpha]$ such that $\alpha' \cap \gamma$ and $\alpha' \cap \mon(\gamma)$ are also minimal. The algorithm of the previous paragraph then determines a distinct right position for $(\tau_\alpha' \circ \mon,\gamma)$ (Figure~\ref{fig_algorithm}(f)). Motivated by this observation, we refine the algorithm by adding \emph{all} points of $\gamma \cap \tau_\alpha(\mon(\gamma))$ which would be obtained by running the algorithm for any allowable isotopy of $\alpha$ (Figure~\ref{fig_algorithm}(g)). As we shall see, this simply involves adding a single point for each downward triangle in the original configuration.

\subsubsection{Details of the inductive step}\label{details}

This subsection gives the technical details necessary to make the algorithm work as advertised, and extracts and proves the various properties we will require the algorithm and its result to have.

We begin with some results concerning triangular regions. Let $\gamma, \gamma'$ be properly embedded, non-isotopic arcs in $\pg$, with $\partial(\gamma) = \partial(\gamma')$, isotoped to minimize intersection. Our motivating example, of course, is the case $\gamma' = \mon(\gamma)$. Let $\G$ denote their union. A \emph{vertex} of $\G$ is thus an intersection point $\gamma \cap \gamma'$. Let $\alpha \in SCC(\pg)$ be isotoped so as not to intersect any vertex of $\G$. A \emph{bigon} in this setup is an immersed disc $B$ bounded by the pair ($\alpha$, $\gamma$) or by ($\alpha$, $\gamma'$). Similarly, a \emph{bigon chain} is a set $\{B_i\}_{i=1}^n$ of bigons bounded exclusively by one of the pairs $(\alpha,\gamma), (\alpha,\gamma')$, for which the union $\bigcup\{\alpha \cap \partial(B_i)\}_{i=1}^n$ is a connected segment of $\alpha$ (Figure~\ref{fig_bigons}(a)). Finally, we say points $p,p' \in \alpha \cap \G$ are \emph{bigon-related} if they are vertices in a common bigon chain.

Suppose then that $T$ is a triangular region in this setup, with vertex $v \in \gamma \cap \gamma'$. We define the \emph{bigon collection associated to $T$}, denoted $\mathcal{B}_T$, as the set of points $\{p \in \alpha \cap \G \ | \ p $ is bigon-related to a vertex of $T\}$ (Figure~\ref{fig_bigons}(b)). Note that, if a vertex of $T$ along $\alpha$ is not a vertex in any bigon, the vertex is still included in $\B_T$, so that, for example, the bigon collection associated to an essential triangle $T$ is just the vertices of $T$ along $\alpha$.

Now, orienting $\gamma$ and $\gamma'$ such that the point $v$ is positive, $v$ separates each of $\gamma$ and $\gamma'$ into two segments, which we label $\gamma_+,\gamma_-,\gamma'_+$ and $\gamma'_-$ in accordance with the orientation. For each bigon collection $\B$ associated to a triangle with vertex $v$, we define $\sigma_v(\B) := (|\mathcal{B} \cap\gamma_+|,|\mathcal{B} \cap\gamma_-|,|\mathcal{B} \cap\gamma'_+|,|\mathcal{B} \cap\gamma'_-|) \in (\mathbb{Z}_2)^4$, where intersections numbers are taken mod 2 (Figure~\ref{fig_bigons}(b)).

\begin{figure}[h!]
\centering \scalebox{.7}{\includegraphics{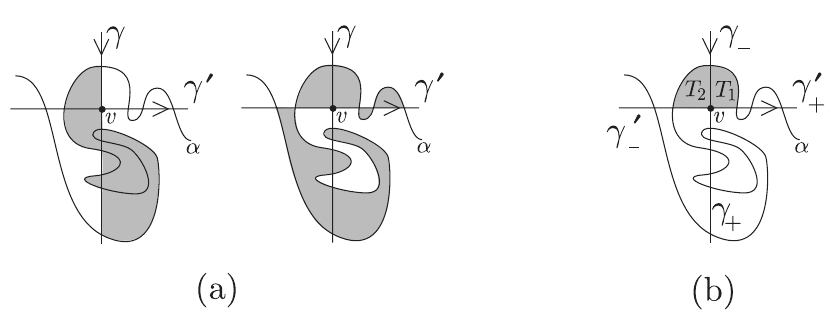}}
\caption[Bigons]{(a) The shaded regions are distinct maximal bigon chains. (b) As each chain from (a) has a vertex in each of the triangular regions $T_1, T_2$, the bigon collection $\B_{T_1} = \B_{T_2}$ includes each intersection point of each chain. In this example, $\sigma_v(\B_{T_1})=(0,1,1,0)$, so the collection is downward. Note that $T_1$ is an essential downward triangular region (Definition \ref{def_triangles}), while $T_2$ is non-essential.}
\label{fig_bigons}
\end{figure}

Finally, we label a bigon collection $\B$ as \emph{upward} (with respect to $v$) if $\sigma_v(\B)$ $ \in \{(1,0,1,0),$ $ (0,1,0,1)\}$, \emph{downward} if $\sigma_v(\B)$ $ \in \{(1,0,0,1),$ $(0,1,1,0)\}$, \emph{non-essential} otherwise (Figure~\ref{fig_bigons}(b)). Thus the bigon collection associated to upward (downward) essential triangular region $T$ is itself upward (downward). 

\begin{lem}\label{lem_triangles}
Let $\alpha \in [\alpha]$ be such that $\alpha \cap \gamma$ and $\alpha \cap \gamma'$ are minimal, and $v \in \gamma' \cap \gamma$. Then the number of essential triangles with vertex $v$ in the triple $(\alpha,\gamma,\gamma')$ depends only on the isotopy class $[\alpha]$.
\end{lem}

\begin{proof}
Let $\alpha$ be an arbitrary representative of $[\alpha]$ (in particular $\alpha \cap \gamma$ and $\alpha \cap \gamma'$ are not necessarily minimal), and $v$ a vertex.  We define an equivalence relation on the set of triangular regions with vertex $v$ by $T_1 \sim T_2 \Leftrightarrow \B_{T_1} = \B_{T_2}$. In particular, there is a 1-1 correspondence between $\{$triangular regions with vertex $v\} \diagup \sim$ and the set of bigon collections associated to triangular regions with vertex $v$.

Now, if $\alpha'$ is another representative of the isotopy class $[\alpha]$, we may break the isotopy into a sequence $\alpha = \alpha_1 \simeq \cdots \simeq \alpha_n = \alpha'$, where each isotopy $\alpha_i \simeq \alpha_{i+1}$ is either a bigon birth, a bigon death, or does not affect either of  $|\alpha \cap \gamma|$ and $|\alpha \cap \gamma'|$. Note that, if $\B$ is upward or downward, an isotopy $\alpha_i \simeq \alpha_{i+1}$ which does not cross $v$ does not affect $\sigma_v(\B)$, while an isotopy $\alpha_i \simeq \alpha_{i+1}$ which crosses $v$ changes $\sigma_v(\B)$ by addition with $(1,1,1,1)$. In either case, the type of the triangle is preserved; i.e. we can keep track of an essential upward (downward) bigon collection through each isotopy. Furthermore, any two distinct bigon collections will have distinct images under each such isotopy (a bigon birth/death cannot cause two maximal bigon chains to merge). We therefore have an integer $a(v)$, defined as the number of essential bigon collections $\B$ associated to triangular regions with vertex $v$, which depends only on $[\alpha]$.

Finally, if we take $\alpha'$ to be a representative of $[\alpha]$ which intersects $\gamma$ and $\gamma'$ minimally, then $\{$triangular regions with vertex $v\} \diagup \sim$ is by definition just the set of essential triangular regions with vertex $v$, and has size $a(v)$.

\end{proof}

\begin{lem}\label{lem_triangle_uniqueness}
Let $\alpha$ be a fixed representative of the isotopy class $[\alpha]$ which has minimal intersection with $\gamma$ and $\gamma'$, and $v \in \gamma \cap \gamma'$ a vertex of an upward triangle $T$. Then $v$ is not a vertex of any downward triangular region $T'$ with edge along $\alpha$.
\end{lem}

\begin{proof}
Consider a neighborhood of $v$, labeled as in Figure~\ref{fig_triangle_intersections}. As $T$ is upward, it must be $T_1$ or $T_3$, while any downward $T'$ must be $T_2$ or $T_4$. Without loss of generality then suppose $T=T_1$. But then neither of $T_2,T_4$ can be triangular regions without creating a bigon, violating minimality.
\end{proof}

\begin{figure}[h!]
\centering \scalebox{.7}{\includegraphics{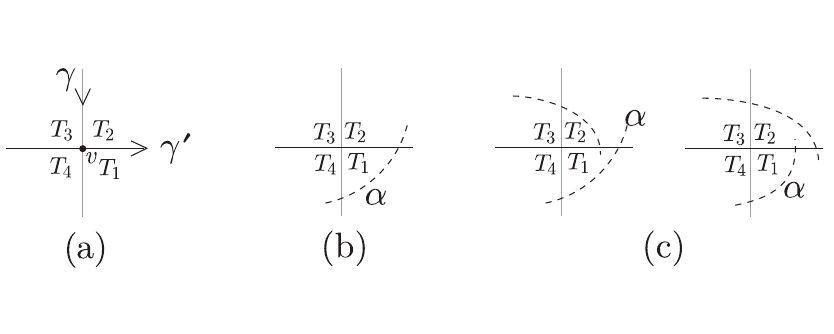}}
\caption[Triangular regions]{(a) A neighborhood of $v$. (b) The upward triangle $T_1$. (c) As $\alpha$ has no self-intersection, $T_2$ cannot be a triangular region without creating a bigon.}
\label{fig_triangle_intersections}
\end{figure}

It follows from Lemmas~\ref{lem_triangles} and \ref{lem_triangle_uniqueness} that we may unambiguously refer to a vertex of an essential triangular region as downward, upward, or neither, in accordance with any essential triangular region of which it is a vertex. In particular, if $\alpha$ has minimal intersection with both $\gamma$ and $\gamma'$, then all triangular regions are essential. Henceforth we will drop the adjective `essential'.

\begin{defin}\label{def_triangle reflection}
Suppose $T$ is a triangular region with vertex $v \in \gamma \cap \gamma'$. Let $h_t^T$, $t\in[0,1]$, denote a family of diffeomorphisms of $\pg$ which isotope $\alpha$ over $T$ to obtain another triangular region $T'$ of the same type (upward/downward) as $T$ with vertex $v$ (Figure~\ref{fig_reflection}). Note that if $T$ contains sub-regions with vertex $v$ then the process involves isotoping the innermost region first, and proceeding to $T$ itself. We call such an isotopy a \emph{shift over $v$}.
\end{defin}

\begin{figure}[h!]
\centering \scalebox{.4}{\includegraphics{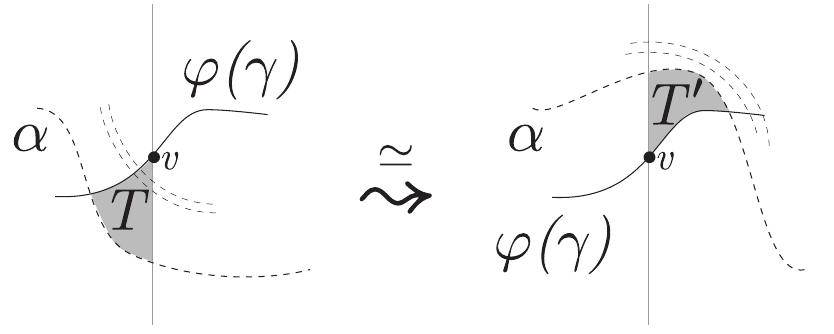}}
\caption[Shifts]{A shift over a vertex $v$.}
\label{fig_reflection}
\end{figure}

We require a final few definitions before we bring the pieces together:

\begin{defin}\label{def_inclusion}
Let $\gamma_1$ and $\gamma_2$ be (not necessarily distinct) properly embedded arcs in $\pg,$ and $\mon \in \mcg$. Suppose $\gamma_1, \mon(\gamma_2)$, and $\alpha$ are isotoped to minimize intersection. Choose $support(D_\alpha)$ to be disjoint from $\gamma_1 \cap \mon(\gamma_2)$. Then let $i_\alpha: (\gamma_1 \cap \mon(\gamma_2)) \hookrightarrow (\gamma_1 \cap D_\alpha(\mon(\gamma_2)))$ be the obvious inclusion. Note that $D_\alpha(\mon(\gamma_2))$ will not in general have minimal intersection with $\gamma_1$.
\end{defin}

\begin{defin}\label{def_fixed_points}
Let $\gamma_1, \gamma_2$ and $\mon$ be as in the previous definition, and  $p \in \gamma_1 \cap
\mon(\gamma_2)$. If there is a representative $\alpha \in SCC(\pg)$ of the isotopy class $[\alpha]$ such that the image $D_\alpha (\mon(\gamma_2))$ can be isotoped to minimally intersect $\gamma_1$ while fixing a neighborhood of $i_\alpha(p)$, we say $p$ is \emph{fixable under} the mapping class $\tau_\alpha$, and $i_\alpha(p)$ simply \emph{fixable}, each with image $p^\alpha$. Of course the identification of $p^\alpha \in \gamma_1 \cap (\tau_\alpha \circ \mon)(\gamma_2)$ depends on $\alpha \in [\alpha]$. Similarly, for a factorization $\omega = \fac$ of $\mon$, we say $p \in \gamma\cap \mon(\gamma_2)$ is \emph{fixable under $\omega$} if $p$ is fixable under each successive $\tau_{\alpha_i}$. Various examples are gathered in Figure~\ref{fig_fixed_points}.
\end{defin}

\begin{figure}[h!]
\centering \scalebox{.9}{\includegraphics{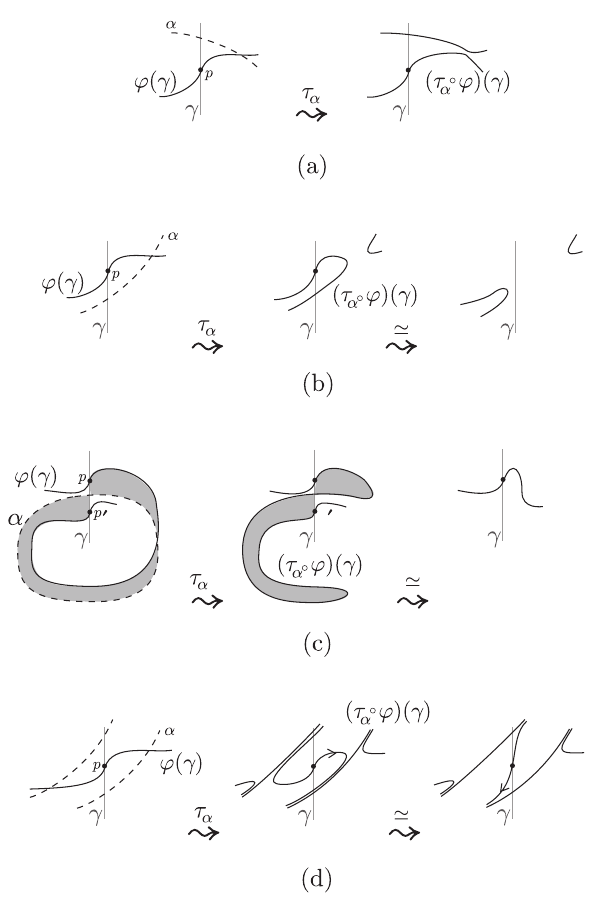}}
\caption[Fixed points]{(a) A downward point is always fixable. (b) An upward point is \emph{not} fixable, unless (c) there is another upward triangle sharing the other two vertices, corresponding to a cancelling bigon in the image. Note that, in this case, while either of the points $v, v'$ is individually fixable, the pair is not necessarily simultaneously fixable. Finally, (d) illustrates the reason for demanding that a \emph{neighborhood} of the point be fixed - though the intersection-minimizing isotopy may be done without removing the intersection point $p$, there is no fixable neighborhood, and so $p$ is not fixable.}
\label{fig_fixed_points}
\end{figure}

\begin{defin}
Let $\gamma_1, \gamma_2$, $\mon$, and $p$ be as in the previous definition. The \emph{image set} of $p$ under $\tau_\alpha$ is $i_\aclass(p) : = \{q \in \gamma_1 \cap (\tau_\alpha \circ \mon)(\gamma_2) \ | \ \exists \alpha \in [\alpha]$ such that $q = p^\alpha\} = \bigcup_{\alpha \in [\alpha]} p^\alpha$.

\end{defin}

Finally,
\begin{defin}\label{p_alpha}

 Let $\pos=\pos(\mon,\gamma)$ be a right position, and $\alpha \in SCC(\pg)$. We define the right position $\pos^\alpha= \pos^\alpha(\tau_\alpha \circ \mon,\gamma)$ by 
 $\pos^\alpha := \bigcup_{p \in \pos} i_{[\alpha]}(p)$.  If $\omega = \fac$ is a word in positive Dehn twists, we define $\pos_{\omega}(\gamma)$ to be the right position given by applying this construction for each twist; i.e. 
 $\pos_{\omega}(\gamma) = (\cdots(((\pos_{\tau_{\alpha_1}})^{\alpha_2})^{\alpha_3})\cdots )^{\alpha_n}$ (where $\pos_{\tau_{\alpha_1}}$ is the right position defined in \textbf{base step}).
\end{defin}

The right position $\pos_\omega(\gamma)$ is thus a maximal subset of the positive intersections in $\gamma \cap \mon(\gamma)$ such that each point of $\pos_\omega(\gamma)$ is fixable under $\omega$, and is in fact the unique such position.

As a simple example, Figure~\ref{fig_lantern_example} illustrates distinct right positions of a pair $(\mon,\gamma)$ associated to distinct factorizations of a mapping class $\mon$.

\begin{figure}[h!]
\centering \scalebox{.7}{\includegraphics{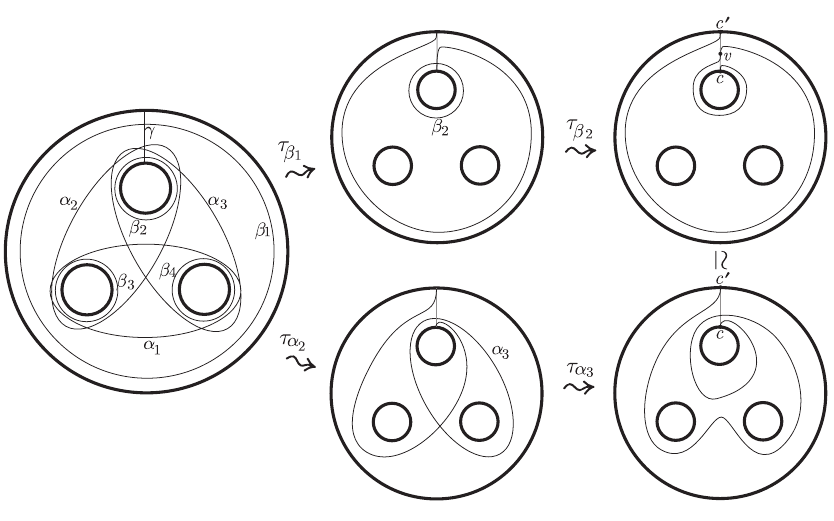}}
\caption[Lantern example]{On the left are the curves of the well-known lantern
relation on $\pg_{0,4}$ - setting $\omega_1 =  \tau_{\beta_4}\tau_{\beta_3}
\tau_{\beta_2} \tau_{\beta_1}  , \omega_2 =  \tau_{\alpha_3}
\tau_{\alpha_2} \tau_{\alpha_1}$, the lantern relation tells us that
$\omega_1, \omega_2$ are factorizations of a common $\mon \in \mcg$.
Clearly, for the arc $\gamma$ indicated, the intersection point $v \in \mon(\gamma) \cap \gamma$ in the upper right figure is fixable under $\omega_1$. However, as is clear from the lower sequence of figures, $v$ is \emph{not} fixable under the factorization $\omega_2$. Thus the right position $\pos_{\omega_1}(\gamma)$ contains the point $v$ as well as the endpoints of $\gamma$, while $\pos_{\omega_2}(\gamma)$ contains only the endpoints of $\gamma$.}
\label{fig_lantern_example}
\end{figure}

\subsection{Consistency of the associated right positions}

The goal of this subsection is to show that the right positions $\pos_\omega(\gamma_1)$ and $\pos_{\omega}(\gamma_2)$ associated by the algorithm of this section to a positive factorization $\omega$ and properly embedded arcs $\gamma_1$ and $\gamma_2$ are consistent (Definition~\ref{def_consistent}).

We require an understanding of the image set of an arbitrary point $p \in \gamma_1 \cap \mon(\gamma_2)$ under a Dehn twist $\tau_\alpha$. We begin by choosing a representative $\alpha$ of the isotopy class $[\alpha]$ which has minimal intersection with $\gamma_1$ and $ \mon(\gamma_2)$. Choose $support(D_\alpha)$ to be disjoint from $\pos$. A general observation of which we will make extensive use is that, for a fixable point $p$, the identification of $p^\alpha$ in the image is unchanged by isotopies of $\alpha$ which do not cross $p$. To be precise:

\begin{lem}\label{lem_alpha_shift}
Let $p \in \gamma_1 \cap \mon(\gamma_2)$, and $\alpha \in SCC(\pg)$, be such that $p$ is fixable under $\tau_\alpha$. If $h_t:\pg \rightarrow \pg$ is an isotopy such that $h_0 = Id$, and $h_t\alpha$ does not cross $p$, then $i_{h_1\alpha}(p)$ is also fixable, and $p^\alpha = p^{h_1\alpha}$
\end{lem}

\begin{proof}
By assumption, there is an isotopy $g_t$, supported away from $i_\alpha(p)$, such that $g_0 = Id$ and $g_1D_\alpha(\mon(\gamma_2))$ has minimal intersection with $\gamma_1$. Thus $g_t \circ D_{h_{1-t}\alpha}(\mon(\gamma_2))$ is an isotopy from $D_{h_1\alpha}(\mon(\gamma_2))$ to $g_1D_\alpha(\mon(\gamma_2))$ which fixes $i_{h_1\alpha}(p)$. The second statement, that the image is unchanged, is immediate.
\end{proof}

 Now, if $p$ is neither downward nor upward in $(\alpha,\gamma_1,\mon(\gamma_2))$, then $p$ is clearly fixable, with a unique image point. Our interest then lies in downward/upward points.

\subsubsection{The image set of a downward point} 
Let $p \in \gamma_1 \cap \mon(\gamma_2)$ be downward with respect to $\alpha$. As $p$ is a vertex of no bigon in $\gamma \cap D_\alpha((\mon(\gamma_2))$, we see that $p^\alpha$ is a unique point whose identification within the image set of $p$ will hold for any isotopy of $\alpha$ which does not involve a shift over $p$. Moreover, each such shift will change the identification, so that if $p$ is the vertex of $m$ distinct downward triangles, $i_\aclass(p)$ will contain exactly $m+1$ points (Figure~\ref{fig_rp_downward}).

\begin{figure}[h!]
\centering \scalebox{.8}{\includegraphics{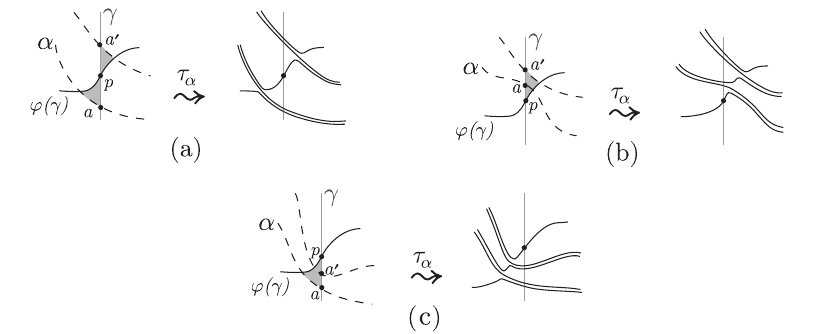}}
\caption[Downward images.]{The point $p$ is downward in 2 distinct triangles. (a), (b), and (c) show representatives of the possible shift-isotopy classes, each of which gives a distinct point $i_\alpha(p)$, while $i_{[\alpha]}(p)$ is the same for each.}
 \label{fig_rp_downward}
\end{figure}

\subsubsection{The image set of an upward point} 

For the case that $p \in \gamma_1 \cap \mon(\gamma_2)$ is upward in $(\alpha,\gamma_1,\mon(\gamma_2))$, assume for the moment that $\alpha$ has been isotoped such that the triangular region $T_p$ is embedded. Note then that $T_p$ corresponds to a bigon $R$ in the pair $\gamma_1,D_\alpha(\mon(\gamma_2))$ in which $i_\alpha(p)$ is a vertex (Figure~\ref{fig_upward_isotopy} (a),(b)). In particular, $i_\alpha(p)$ is fixable only if there is an second bigon $R'$, which intersects $R$ in the vertex which is not $i_\alpha(p)$ (Figure~\ref{fig_upward_isotopy} (c)). Observe then that $R'$ also corresponds to an upward triangle, $T_{p'}$, where $p'\in \gamma_1 \cap \mon(\gamma_2)$, such that $T_p \cap T_{p'}$ (for the given $\alpha$ in its isotopy class) consists of the two vertices along $\alpha$ (Figure~\ref{fig_upward_isotopy} (d)). We refer to such $T_{p'}$ as the \emph{pair} of $T_p$, and say that each of $T_p$ and $T'_p$ is \emph{paired}. Note that each of $i_\alpha(p)$ and $i_\alpha(p')$ are thus fixable, though not simultaneously, and have the same image. If $i_\alpha(p)$ is \emph{not} fixable, we refer to $T_p$ as \emph{unpaired}. Note further that, while the property of being paired depends on the representative $\alpha$, using Lemma~\ref{lem_alpha_shift}, a paired region on a point $p$ will remain paired for any isotopy of $\alpha$ which does not cross $p$ (though the regions may cease to be embedded).

\begin{figure}[h!]
\centering \scalebox{.9}{\includegraphics{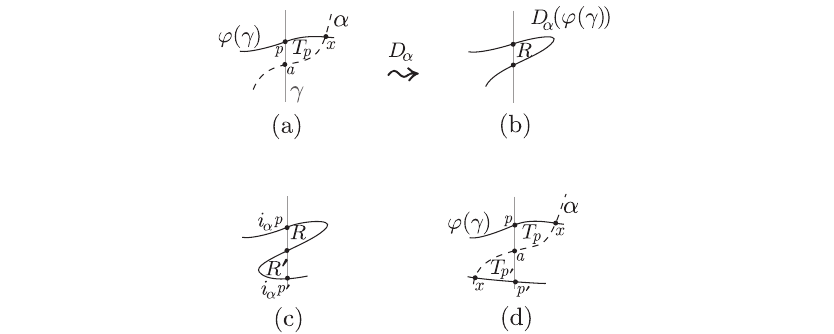}}
\caption[Downward images.]{(a) An upward point, and (b) a corresponding bigon in the image. (c) Cancelling bigons, and (d) corresponding paired regions.}
\label{fig_upward_isotopy}
\end{figure}

So, we see that for given $\alpha$, and fixable point $p$, $i_\alpha(p)$ will be fixable if and only if each upward triangle on $p$ is paired. This motivates the following definition:

\begin{defin}\label{def_nice}
Let $\{\gamma_i\}$ be a collection of properly embedded arcs in $\pg$, $\mon \in \mcg$, and $\alpha \in SCC(\pg)$. Then $\alpha$ is \emph{nice} (with respect to the arcs $\{\gamma_i\}$ and $\{\mon(\gamma_i)\}$) if, whenever $p \in \gamma_i \cap \mon(\gamma_j)$ is fixable, then $i_\alpha(p)$ is fixable under some isotopy of $D_\alpha(\mon(\gamma_j))$ which eliminates all bigons with $\{\gamma_i\}$ (Figure~\ref{fig_nice}).
\end{defin}

\begin{figure}[htb]
\centering
\includegraphics[scale=.8]{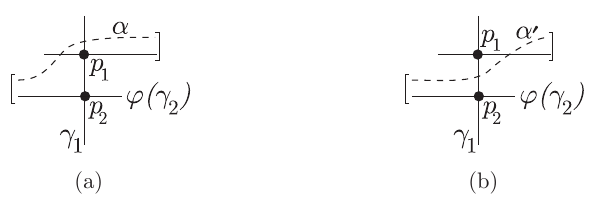}
\caption[]{(a) The curve $\alpha$ is not nice with respect to $\gamma_1$ and $\mon(\gamma_2)$, as $i_\alpha(p_1)$ is not fixable. (b) The situation is remedied by a shift over $p_1$.}
\label{fig_nice}
\end{figure}

So, if $\alpha$ is nice, then for each fixable $p$, $i_\alpha(p)$ is fixable. As it turns out, niceness it not a particularly demanding requirement. Indeed:

\begin{lem}\label{lem_nice}

Let $\alpha \in SCC(\pg)$, $\mon \in \mcg$, and $\{\gamma_i\}$ a collection of properly embedded arcs. Suppose $P \subset \{\gamma_i \cap \mon(\gamma_j)\}$ is a collection of upward, fixable points such that for each $p \in P$, $i_\alpha(p)$ is fixable. Then there is nice $\alpha' \in [\alpha]$ such that for each $p \in P$, $i_{\alpha'}(p)=i_\alpha(p)$; i.e. all identifications are preserved.
 
\end{lem}

\begin{proof}

If $\alpha$ is not nice, then there is some upward fixable point $p_1$ such that the associated triangle $T_{p_1}$ is not paired. We may assume $T_{p_1}$ is an innermost such region. Using Lemma~\ref{lem_alpha_shift}, we want to show that $T_{p_1}$ contains no upward fixable point, so that the shift over $T_{p_1}$ changes no identifications. Suppose otherwise, so there is upward fixable $p_2 \in T_{p_1}$, with associated $T_{p_2} \subset T_{p_1}$ (Figure~\ref{fig_paired}(a)); as we are assuming $T_{p_1}$ is innermost among unpaired regions, $T_{p_2}$ is paired (Figure~\ref{fig_paired}(b)). To keep track of things, we index the involved arcs such that $p_1 \in \gamma_1 \cap \mon(\gamma_2)$ and $p_2 \in \gamma_3\cap \mon(\gamma_4)$.

Let $T_{p_2'}$ be the pair of $T_{p_2}$, so $p'_2 \in \gamma_3 \cap \mon(\gamma_4)$. Now, $T_{p_2}$ is such that each of the 2 vertices along $\alpha$ of $T_{p_1}$ lies in the interior of the edge along $\alpha$ of $T_{p'_2}$ (Figure~\ref{fig_paired}(b)). Thus following along $\gamma_1$ and $\mon(\gamma_2)$ away from $p_1$ into $T_{p'_2}$, they must intersect in some point $p'_1$, thus giving a region $T_{p'_1} \subset T_{p'_2}$ which is paired with $T_{p_1}$  (Figure~\ref{fig_paired}(c)), a contradiction.

We may then shift $\alpha$ over $T_{p_1}$, and repeat the process with any remaining innermost unpaired regions to obtain nice $\alpha'$ as desired.

\end{proof}

\begin{figure}[htb]
\centering
\includegraphics[scale=.8]{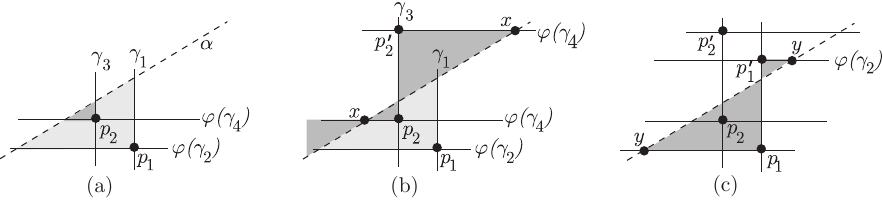}
\caption[]{(a) Region $T_{p_1}$ is lightly shaded, $T_{p_2}$ is dark. (b) The pair of regions $T_{p_2}$ and $T_{p_2'}$ are darkly shaded. As usual, the figures are in the universal cover, with all lifts of a given object given the same label as the object. (c) The region pair  $T_{p_1}$ and $T_{p'_1}$ is shaded.}
\label{fig_paired}
\end{figure}

A further notion we will require is that of a \emph{reflection}:

\begin{remark}\label{rem_reflection}
Given $\obd$ such that $\mon$ has factorization $\fac$, consider $(\refpg,\refmon)$, where $\refpg$ is gotten by reversing the orientation of $\pg$, and $\refmon$ has factorization $\reffac$. We refer to $(\refpg,\refmon)$ as the \emph{reflection} of $\obd$. Note that, e.g. if $p\in \gamma_1 \cap \mon(\gamma_2)$ is upward (for some $\alpha$), then its reflection is downward. Furthermore, given $p \in \gamma_1 \cap \mon(\gamma_2)$ and $q \in \gamma_1 \cap \tau_\alpha(\mon(\gamma_2))$, we have  $q \in i_\aclass(p)$ if and only if the reflection of $p$ is in the $\alpha-$image set of the reflection of $q$ (Figure~\ref{fig_reflections}(a)). Finally, if we have right positions $\pos_i$ for some collection $\{(\mon,\gamma_i)\}$, and associated initially parallel regions $\R(\pos_j,\pos_k)$ for some $j$ and $k$, we obtain a collection of regions in the reflection by simply switching \cpoints and \dpoints; i.e. if $p$ is an \dpoint (\cpoint) of a region, its reflection is a \cpoint (\dpoint). We denote the reflection of a region $B$ by $\refB$, and decorate the vertices accordingly (Figure~\ref{fig_reflections}(b)).  
\end{remark}

\begin{figure}[htb]
\centering
\includegraphics[scale=1.3]{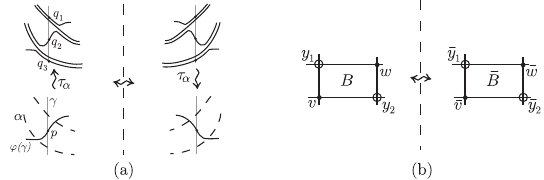}
\caption[]{(a) A point $p \in \gamma \cap \mon(\gamma)$, its image set $i_\aclass(p)$, and the reflections of each of these. (b) Reflecting a region.}
\label{fig_reflections}
\end{figure}

\begin{defin}\label{def_alpha_equiv}

Motivated by the above, given nice $\alpha$, we introduce an equivalence relation $\aequiv$ on intersection points $\gamma_1 \cap \mon(\gamma_2)$, where $p \aequiv p' \Leftrightarrow p^\alpha = p'^\alpha$ (Figure~\ref{fig_alpha_equiv}). Associated to each such equivalence class $[p]_\alpha$ are then the unique minimal connected sub-arcs of $\gamma_1$ and $\mon(\gamma_2)$ which contain each point of the class; we refer to these as the \emph{vertical} (along $\gamma_1$) and \emph{horizontal} (along $\mon(\gamma_2)$) spans of the class.  Observe that the image under $D_\alpha$  of the horizontal span is entirely contained in the boundary of a bigon chain of even length in $D_\alpha(\mon(\gamma_2))$ and $\gamma_1$.

\end{defin}

\begin{figure}[htb]
\centering
\includegraphics[scale=1.9]{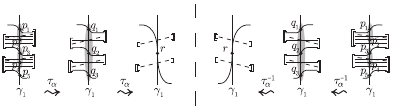}
\caption[]{Equivalence classes are shaded. On the left side, we have $i_\alpha(p_1)=i_\alpha(p_2) = q_1, i_\alpha(p_3)=i_\alpha(p_4) = q_2$, and $i_\alpha(p_5) = q_3$, while $i_\alpha(q_i)=r$ for each $i$. On the right side, their reflections.}
\label{fig_alpha_equiv}
\end{figure}

\begin{lem}\label{lem_cycle_config} For an arc $\sigma$, and points $a,b \in \sigma$, let $[a,b]_\sigma$ denote the (closed) segment of $\sigma$ from $a$ to $b$.
Suppose $p,p' \in \gamma_1 \cap \mon(\gamma_2)$ are consecutive (along $\gamma_1$) elements of the equivalence class $[p]_\alpha$. Then nice $\alpha$ intersects each of $[p,p']_{\gamma_1}$ and $[p,p']_{\mon(\gamma_2)}$ exactly once. 
\end{lem}

\begin{proof}
As $p$ is fixable, we may isotope $\alpha$ to $\alpha'$, such that there are paired, in particular embedded, triangles $T_p$ and $T_{p'}$ on the given points. Thus $\alpha'$ satisfies the desired intersection conditions, and by Lemma~\ref{lem_nice}, may be assumed nice. In particular, $p \stackrel{\alpha'}{\sim} p'$. Now, the isotopy $\alpha \leadsto \alpha'$ can involve a shift over at most one of $T_p$ and $T_{p'}$, but either of these will change the image under $i_{\alpha}$ of exactly one of the points, so that $p$ and $p'$ are no longer in the same equivalence class, a contradiction. So the isotopy crosses neither point, and so $\alpha$ satisfies the same intersection conditions.
\end{proof}

Now, the Lemma~\ref{lem_cycle_config} implies in particular that, if $p,p' \in \gamma_1 \cap \mon(\gamma_2)$ are consecutive (along $\gamma_1$) elements of the equivalence class $[p]_\alpha$, then the closed curve obtained as the union of $[p,p']_{\gamma_1}$ and $[p,p']_{\mon(\gamma_2)}$ along the endpoints is isotopic to $\alpha$ through an isotopy which crosses neither $p$ nor $p'$. Suppose then that there is a further point $p''$ such that the union of $[p',p'']_{\gamma_1}$ and $[p',p'']_{\mon(\gamma_2)}$ along the endpoints is a (nice) curve $\alpha'$ again isotopic to $\alpha$. Then either $\alpha' \cap [p,p']_{\gamma_1} \neq \emptyset$, so that $p'' \in [p]_\alpha$, or otherwise, so there is an isotopy $\alpha \leadsto \alpha'$ given by a shift over a triangular region with vertex $p'$. Using this observation, we have:

\begin{lem}\label{lem_cycles}
If $p \in \gamma_1 \cap \mon(\gamma_2)$, and $\alpha$ a nice representative of $[\alpha]$, are such that $[p]_\alpha$ contains greater than 2 elements, then any nice $\alpha' \in [\alpha]$ is isotopic to $\alpha$ through an isotopy which intersects no element of $[p]_\alpha$; i.e. $[p]_\alpha = [p]_{\alpha'}$.
\end{lem}

\begin{flushright}$\square$

\end{flushright}

As an example of Lemma~\ref{lem_cycles}, consider again Figure~\ref{fig_alpha_equiv}. Observe that $\alpha$ may be isotoped over $p_2$ and $p_4$ to a nice representative $\alpha'$, thus changing the equivalence relation to $p_2 \sim p_3$ and $p_4 \sim p_5$. In the second picture (from either side), however, there can be no nice representative of $\alpha$ which does not preserve $q_1 \sim q_2 \sim q_3$.

\begin{defin}\label{def_collapsed}
Let $B \in \R(\pos_1,\pos_2)$ and $\alpha \in SCC(\pg)$ a fixed, nice, representative of its isotopy class. We say $B$ is \emph{collapsed} (by $\alpha$) if (1) each edge is contained in the span of the $\alpha-$equivalence class of its \cpoint vertex, and (2) $\alpha \cap B \neq \emptyset$. Similarly, a region $\preB \in \R(\pos^\alpha_1,\pos^\alpha_2)$ is \emph{created} (by $\alpha$) if its reflection is collapsed (Figure~\ref{fig_collapsed_created}).

\end{defin}

\begin{figure}[htb]
\centering
\includegraphics[scale=1.7]{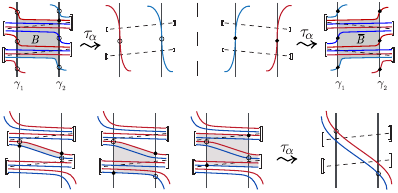}
\caption[]{Above: To the left, a collapsed region $B$, and its (trivial) image. To the right, the reflection of $B$ is a created region $\refB$. Below: The three figures to the left are various regions in a common configuration, which is mapped by $\tau_\alpha$ to the right-most figure. From left to right, these 3 regions are non-collapsed (satisfies condition (1) but not condition (2)), non-collapsed (satisfies condition (2) but not condition (1)), and collapsed (but not embedded)}
\label{fig_collapsed_created}
\end{figure}

We have:
 
\begin{lem}\label{lem_collapsed}
Suppose $B \in \R(\pos_1,\pos_2)$ (not necessarily consistent) is collapsed by some $\alpha$. Then there is $B^c \in \R(\pos_1,\pos_2)$ with the same \dpoints as $B$, and opposite orientation. Similarly, if $\preB \in \R(\pos^\alpha_1,\pos^\alpha_2)$ is created by $\alpha$ then there is ${\preB}^c \in \R(\pos^\alpha_1,\pos^\alpha_2)$ with the same \cpoints as $\preB$, and opposite orientation.
\end{lem}

\begin{proof}
We begin with the first statement. We will find it convenient to label the edges of $B$ as $e_i$, $i=1,2,3,4$, where $e_1$ denotes the edge along $\mon(\gamma_2)$, $e_1$ the edge along $\mon(\gamma_2)$, and for $i=3,4$, $e_{i}$ is the edge along $\gamma_{i-2}$. We begin by assigning a \emph{length} $l(B) = \#|[y_1]_\alpha \cap e_1 |$ and \emph{width} $w(B) = \#|[y_1]_\alpha \cap e_3 |$ (so, e.g., the region $B$ in the upper left corner of Figure~\ref{fig_collapsed_created} has length 1, width 2, while the collapsed but not embedded example on the lower row has length=width=1). Furthermore, we will order points along a given oriented arc in accordance with that orientation, so if $p$ and $p'$ are points along oriented arc $\sigma$, then $p' >_\sigma p$ means that the orientation of $\sigma$ points from $p$ to $p'$. 

We first consider the case $l=w=1$, and orient all arcs such that $\partial B$ has the standard (counterclockwise) orientation (Figure~\ref{fig_collapsed_lemma}(a)). By the collapsed condition, $[y_1]_\alpha$ contains at least one element greater than $v$ along $\gamma_1$; let $y_1'$ denote the minimal such element. As $w=l=1$, we see that $y'_1$ and $y_1$ are consecutive in the class, and so $v \in [y_1,y'_1]_{\gamma_1}$, and $w \in [y_1,y'_1]_{\mon(\gamma_2)}$. In particular, $y_2$ is a vertex of an embedded upward triangular region $T_{y_2}$; by niceness of $\alpha$ there is then $y'_2$ and embedded pair $T_{y'_2}$. Furthermore, $v \in [y_2,y'_2]_{\mon(\gamma_1)}$, and $w \in [y_2,y'_2]_{\gamma_2}$. Now, observe that, by isotoping $\alpha$ to pass through each \dpoint, we may realize $B$ as $T_{y_1} \cup T_{y_2}$. But then the union of the complementary regions $T_{y'_1}$ and $T_{y'_2}$ give a region $B^c$ as desired.

\begin{figure}[htb]
\centering
\includegraphics[scale=1.9]{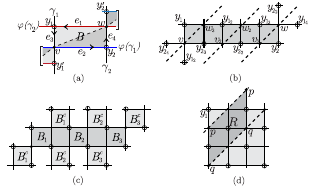}
\caption[]{}
\label{fig_collapsed_lemma}
\end{figure}

Consider then the case $l(B)>1, w(B)=1$ (Figure~\ref{fig_collapsed_lemma}(b)). Let $y_{1_i} \in [y_1]_\alpha$, $i=1,\ldots,l+1$, be consecutive elements of the class, with $y_{1_1}=y_1$, and $y_{1_2} \in e_1$. As $w=1$, we have $y_{1_2} >_{\gamma_1} v$. Now, each $y_{1_i}$, $1 \leq i \leq l$, is the endpoint of a segment of $\gamma_1$ which cuts across $B$; let $v_i$ denote the opposite (on $e_2$) endpoint of this segment (so $v_1 = v$). For notational convenience, we augment $\pos_1$ to $\pos'_1$ by including the $v_i$. Similarly, we have $y_{2_i} \in [y_1]_\alpha$, $w_i$ along $e_2$, and $\pos'_2 = \pos_2 \cup \{w_i\}$. There are then $B_i \in \R(\pos_1',\pos_2')$, $i=1,\ldots,l$, where each $B_i$ is a subregion of $B$ with vertices $y_{1_i},v_i,y_{2_{l+1-i}}$ and $w_{l+1-i}$ (Figure~\ref{fig_collapsed_lemma}(c)). In particular, each $B_i$ is collapsed, and $l(B_i)=w(B_i)=1$, so by the previous paragraph there are complementary regions $B^c_i$. But then the union of the regions $B^c_i$ with the complement in $B$ of the $B_i$ is a region $B^c$ as desired. Observe that $w(B^c)=l(B)$, and $w(B)=l(B^c)$; in particular, the argument for the case $w(B)>1, l(B)=1$ is identical, with the roles of the $B_i$ and $B_i^c$ swapped.

Finally, suppose both $l(B)<1$ and $w(B)<1$. Let $p$ denote the (unique) intersection of $\alpha$ with $[y_1,y_{1_2}]_{\gamma_1}$, $q$ the intersection with $[y_{1_2},y_{1_3}]_{\gamma_1}$. Then, lifting to the universal cover (Figure~\ref{fig_collapsed_lemma}(d)), we find a rectangular region $R$ bounded by $\rho^{-1}(\alpha)$ and $\rho^{-1}(\gamma_1)$, and whose vertices are consecutive (along $\rho^{-1}(\alpha)$) lifts of $p$ and $q$ (where $\rho: \w{\pg} \mapsto \pg$ is the covering map). But then $\rho(R)$ is an annulus in $\pg$ bounded by two copies of $\alpha$, a contradiction.

The second statement then follows immediately by reflection: if $B$ is created by $\alpha$, it is the reflection of a collapsed (by $\alpha$) region $\refB$. There is thus $\refB^c$ as above, whose reflection $B^c$ therefore shares the \cpoints of $B$, and has the opposite orientation, as desired.

\end{proof}

\subsubsection{Consistency} 
We now have the necessary tools and terminology to demonstrate Theorem~\ref{thm_rp}. As such, let  $\gamma_1$ and $\gamma_2$ be properly embedded arcs in $\pg$, $\mon \in \mcg$, with right positions $\pos_1$ and $\pos_2$. Our goal then is to show that, if each region $B \in \R(\pos_1,\pos_2)$ is completed, the same is true of each region in $\R(\pos^\alpha_1,\pos^\alpha_2)$.

Many of the arguments of this subsection will involve keeping track of equivalence classes of upward vertices. An observation which will prove particularly useful throughout is the following:

\begin{lem}\label{lem_rectangle}
Given $\obd$, and $\alpha \in SCC(\pg)$, let $\{\gamma_i\}$ be a collection of properly embedded arcs in $\pg$. Suppose $p,p' \in \gamma_j \cap \mon(\gamma_k)$ are consecutive in $[p]_\alpha$, and further that $q,q' \in \gamma_l \cap \mon(\gamma_m)$ are such that $p,p',q$ and $q'$ are the vertices of a rectangular region (so either $j=l$, or $k=m$). Then there is $\alpha' \in [\alpha]$ such that $q$ and $q'$ are consecutive in $[q]_{\alpha'}$.
\end{lem}

\begin{proof}
While this is obvious, the point of view which makes it so perhaps bears reinforcement. As such, consider the case $k=m$ (so the edges $[p,p']$ and $[q,q']$ of the given rectangular region are along elements of $\{\gamma_i\}$). Now, as $p,p'$ are consecutive in $[p]_\alpha$, we have a pair of upward triangles $T_p$ and $T_{p'}$, which have in common 2 vertices, one along $\gamma_j$, and one, which we label $x$, along $\mon(\gamma_k)$. As usual, we picture the setup in the universal cover (Figure~\ref{fig_rectangle} (a)). Expanding our view of the figure in the universal cover to include 2 copies of our original figure (Figure~\ref{fig_rectangle} (b)), we find that $q$ and $q'$ are themselves vertices in upward triangular regions, and that using these regions $\alpha$ may be isotoped to some nice $\alpha'$ for which $q$ and $q'$ are consecutive in $[q]_{\alpha'}$ (Figure~\ref{fig_rectangle} (c)). The modifications necessary for the case $j=l$ are clear (Figure \ref{fig_rectangle} (d)).

\end{proof}

\begin{figure}[htb]
\centering
\includegraphics[scale=1.1]{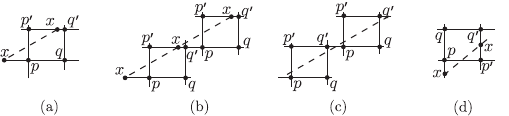}
\caption[]{Figures for Lemma~\ref{lem_rectangle}.}
\label{fig_rectangle}
\end{figure}

Now, recall that, as in the discussion at the beginning of the previous sub-section, if some intersection $p \in \gamma_1 \cap \mon(\gamma_2)$ is both upward and fixable for a given $\alpha$, then there is some $p' \in [p]_\alpha$, and upward triangular regions $T_p$ and $T_{p'}$ such that the entirety of $\alpha$ is contained in their edges; it follows that in this case every triangular region of $(\alpha,\gamma_i,\mon(\gamma_j))$, $i,j \in \{1,2\}$ is upward. As such, we refer to a region with fixable vertices as \emph{upward} if some vertex is in an upward triangular region.

Our argument depends on being able to consider a given region of $\R(\pos^\alpha_1,\pos^\alpha_2)$ as the `image' under $\tau_\alpha$ of a region in (possibly an augmentation of) $\R(\pos_1,\pos_2)$. Now, it is clear intuitively that a region should not have an image if it is collapsed, or if its vertices are not all fixable under $\tau_\alpha$. We will further impose a minimality condition, to obtain:

\begin{defin}
Given $(\pg,\mon)$, and a pair $\gamma_1$ and $\gamma_2$ of disjoint, properly embedded arcs in $\pg$, suppose $\pos_i$, $i=1,2$ are right positions for $(\mon, \gamma_i)$, and $\alpha \in SCC(\pg)$. Let $B \in \R(\pos_1,\pos_2)$ be non-collapsed (for $\alpha$), and have fixable \cpoints. Then $B$ is a \emph{preimage for $i_\alpha$} if 

\begin{enumerate}

\item $B$ is not upward, and no downward triangular region on either \dpoint is contained in $B$, or

\item $B$ is upward, and for $i=1,2$, the equivalence class $[y_i]_\alpha$ contains no point in the interior of an edge of $B$.

\end{enumerate}
\end{defin}

\begin{defin}
Given  $B \in \R(\pos_1,\pos_2)$, we refer to any region of $\R(\pos^\alpha_1,\pos^\alpha_2)$ whose vertices are each in the respective image sets of the vertices of $B$ as an \emph{image} of $B$.

\end{defin}

As some justification for this terminology, we have:

\begin{lem}\label{lem_region_image}
Let $B \in \R(\pos_1,\pos_2)$ be a preimage for $i_\alpha$. Then there is an image $\preB$ of $B$, such that $i_\alpha$ maps the \cpoints of $B$ to those of $\preB$.

\end{lem}
\begin{proof}
To simplify the picture, let $\w{B}$ be a lift of $B$ to the universal cover $\w{\pg}$, and $\{\w{\alpha}_i\}$ the connected components of $\rho^{-1}(\alpha)$ (where $\rho:\w{\pg} \rightarrow \pg$ is the covering map). The lifted map $\w{D_\alpha}$ has disjoint support, each connected component of which contains some $\w{\alpha}_i$; let $D_{\w{\alpha}_i}$ denote the restriction of $\w{D_\alpha}$ to this component of the support. Now, a given $D_{\w{\alpha}_i}$ affects $\w{B}$ non-trivially only if $\w{\alpha}_i$ has non-empty intersection with at least one of the edges $\w{e_1}$ and $\w{e_2}$ (recall that $e_1$ denotes the edge of $B$ along $\mon(\gamma_2)$, $e_2$ the edge along $\mon(\gamma_1)$). These $\w{\alpha}_i$ can then be distinguished as those which intersect exactly one of the edges $\w{e_1}$ and $\w{e_2}$, which we refer to as a \emph{corner} intersection, and those which intersect both edges, which we refer to as a \emph{vertical} intersection. 

Now, if $\w{\alpha}_i$ has a vertical intersection, $D_{\w{\alpha}_i}$ will slice each lift of $B$, and reattach to form new regions whose vertices map to the correct vertices under $\rho$ (Figure~\ref{fig_upward}(a)). We are then done as long as no vertex is contained in a bigon in the resulting configuration. This however follows easily from our assumptions: Suppose otherwise, so $\w{\alpha}_i \cap \w{B}$ is contained in an edge of an upward triangular region, say on $\w{y}_1$. Then for fixability there is $y'_1 \in [y_1]_\alpha$, and by condition (2) we have each edge with endpoint $y_1$ contained in its span (Figure~\ref{fig_upward}(b)). Thus $B$ is collapsed, a contradiction.

\begin{figure}[h!]
\centering \scalebox{1.1}{\includegraphics{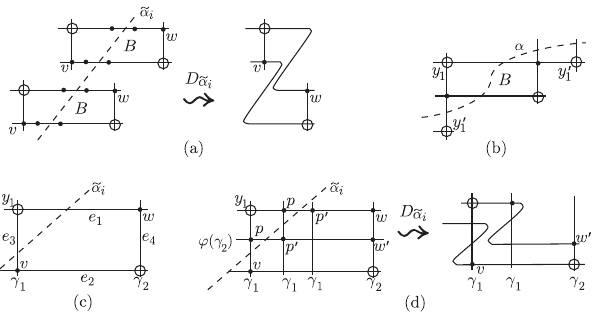}}
\caption{(a): A vertical intersection and the effect of the corresponding $D_{\w{\alpha}_i}$ on the lifted region. (b) If a vertical segment of $\alpha \cap B$ is contained in an upward triangle, then $B$ cannot be a pre-image. (c) A corner intersection, and (d) the effect of $D_{\w{\alpha}_i}$}.
\label{fig_upward}
\end{figure}

Consider then an (upward) corner intersection; such $\w{\alpha}_i$ is upward on a \cpoint, say $\w{y_1}$ (Figure~\ref{fig_upward}(c)). As $y_1$ is fixable, and $B$ not collapsed, there are then points $p, p' \in \gamma_1 \cap e_1$, $p<_{\gamma_1}p'$ consecutive in $[p]_\alpha$, such that $\w{\alpha} \cap [\w{p},\w{p'}]_{\w{e_1}}$ is a single point. As $\w{\alpha}_i$ exits $\w{B}$ through $\w{e_3}$, the point $p$ also lies along $e_3$ (i.e. the edge of $B$ contained in $\gamma_1$) (Figure~\ref{fig_upward}(d)). Now, travelling along $\w{e_1}$ from $\w{y_1}$, turning right onto $\w{\alpha}_i$ and travelling exactly one length of $\alpha$, one thus arrives at a second lift of $\mon(\gamma_2)$, at a point in the interior of $\w{B}$. Letting $\w{w'}$ denote the intersection of this lift with $\w{\gamma_2}$, by Lemma~\ref{lem_rectangle} we have $w' := \rho(\w{w'}) \in [w]_{\alpha'}$. In particular, $i_\alpha(w') \in \pos^\alpha_2$. The map $D_{\w{\alpha}_i}$ cuts each component of $\rho^{-1}(\mon(\gamma_2))$ at its intersection with $\w{\alpha_i}$, and reattaches by inserting copies of (a fundamental domain of) $\alpha$. So, after the cutting, $\w{\mon(\gamma_2)}$ is attached to the lift of $\mon(\gamma_2)$ through $\w{w'}$ via a path in $\w{B}$.

Thus, after performing $\w{D_\alpha}$, we have a region $\w{\preB}$ such that all bigons on the vertices are exterior to the region, and contain no other vertices. We may therefore isotope the $\w{D_\alpha}\rho^{-1}(\mon(\gamma_i))$ so as to eliminate bigons, and preserve $\w{\preB}$. As all isotopic bigon free configurations are isotopic through bigon free configurations, the result is thus equivalent to $\rho^{-1}\tau_\alpha(\mon(\gamma_i))$, and so $\rho(\w{\preB}) = \preB \in \R(\pos^\alpha_1,\pos^\alpha_2)$ is a region as desired.

\end{proof}

More generally, we would like to be able to find an image for \emph{any} non-collapsed region with fixable \cpoints. We break the argument into two lemmas, for the cases of upward and non-upward regions.

\begin{lem}\label{lem_upward}
Let $\pos_1$ and $\pos_2$ be right positions. Suppose $B \in \R(\pos_1,\pos_2)$ is a non-collapsed upward region with fixable \cpoints $y_1$ and $y_2$, $\alpha \in SCC(\pg)$, and $\hat{y}_i \in i_\aclass (y_i)$ for $i=1,2$. Then $\hat{y}_1$ and $\hat{y}_2$ are the \cpoints of an image of $B$.
\end{lem}

\begin{proof}
 Let $\pos'_i$ denote the maximal right position which both contains $\pos_i$, and is such that $(\pos'_i)^\alpha = \pos^\alpha_i$. We will show that $\R(\pos'_1,\pos'_2)$ contains a preimage for $i_\alpha$ which maps to the desired region.

By assumption, at least one of the \cpoints, say $y_1$, is upward. We then may fix nice $\alpha$ such that $y_1^\alpha = \hat{y}_1$ and $[y_1]_\alpha$ contains more than one element.

Let $S'$ denote the set of regions of $\R(\pos'_1,\pos'_2)$ whose orientations match that of $B$, and whose \cpoints are \emph{simultaneously} fixable to  to the given $\hat{y}_1$ and $\hat{y}_2$; i.e such that, for some $\alpha' \in [\alpha]$, $i_{\alpha'}$ sends the \cpoints to $\hat{y}_1$ and $\hat{y}_2$. We will show firstly that $S'$ contains a non-collapsed element.

Suppose then that $B \not\in S'$; i.e. for $i=1,2$, the sets $\{\alpha \in [\alpha] \ | \ y_i^\alpha = \hat{y}_i \}$ are disjoint. Thus by Lemma~\ref{lem_cycles}, $[y_1]_\alpha$ has exactly two elements: $y_1$, and a point we label $y'_1$. We denote the upward triangular region on $y_1$ by $T_{y_1}$. Then, as $y_2$ is also fixable, there is upward triangle $T_{y_2}$ and shift isotopy $h_t^{T_{y_2}}$ such that $y_2^{\alpha'} = \hat{y_2}$, where $\alpha'=h_1^{T_{y_2}}\alpha$ (Figure~\ref{fig_s'_a}(a)). As $B \not\in S'$, this isotopy must cross $y_1$, so $y_1$ lies in the interior of $T_{y_2}$. Again by Lemma~\ref{lem_cycles}, the class $[y_2]_{\alpha'}$ also contains exactly 2 points: $y_2$, and a point we label $y'_2$. The points $y_1'$ and $y_2'$ thus lie in the exterior of $T_{y_2}$, so ${y'}_i^{\alpha'} = {y'}_i^{\alpha}=\hat{y}_i$ for either $i$. Now, consider a neighborhood of $y_1$ chosen small enough such that its intersection with $\gamma_1 \cup \mon(\gamma_2)$ is a pair of arcs intersecting at $y_1$, and let $T_\epsilon$ denote the intersection of this neighborhood with $T_{y_1}$. We have two cases:

\begin{enumerate}

\item $T_\epsilon \not\subset B$ (Figure~\ref{fig_s'_a}(b)). Pushing to the universal cover, let $\w{\mon(\gamma_2)}'$ denote the lift of $\mon(\gamma_2)$ through $\w{y}'_1$; this arc enters $\w{T}_{y_2}$ at its intersection with $\w{\alpha}$, so continues to intersect  $\w{\gamma_2}$ at a point we denote $\w{w}'$. Note then that, using Lemma~\ref{lem_rectangle}, $w' := \rho(\w{w}')$ is in $\pos'_2$. Letting $B_1$ denote the region obtained by extending $B$ along the $\gamma_i$ from $w$ and $y_1$ to $w'$ and $y_1'$, and isotoping $\alpha$ to $\alpha'$, we have $B_1 \in S'$ (Figure~\ref{fig_s'_a}(c)). As no edge is contained in the span of the equivalence class of its \cpoint, $B_1$ is not collapsed, as desired. 

\begin{figure}[htb]
\centering
\includegraphics[scale=1.1]{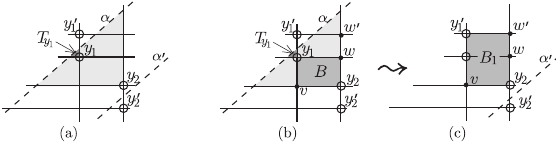}
\caption[]{Figures for Lemma~\ref{lem_upward}. The region $T_{y_2}$ is lightly shaded in each.}
\label{fig_s'_a}
\end{figure}

\item $T_\epsilon \subset B$ (Figure~\ref{fig_s'_b} (a)). As $T_{y_1}$ does not contain $y_2$, it follows that $\alpha \cap B \neq \emptyset$. Note that if $y_1'$ lies along neither $e_1$ nor $e_3$, then each of these is contained in the span of $[y_1]_\alpha$ (Figure~\ref{fig_s'_b} (b)). But then as $\alpha$ may be assumed nice, each of $e_2$ and $e_4$ is contained in the span of $[y_2]_\alpha$, so $B$ is collapsed, a contradiction. 

Thus $y_1'$ lies along $e_1$ or $e_3$ (Figure~\ref{fig_s'_b} (c), (d)). Then, again using Lemma~\ref{lem_rectangle}, $y_1'$ is the \cpoint of a subregion $B_1$ in $S'$, whose edge along $\mon(\gamma_2)$ is not contained in the span of $[y_1']_\alpha$, so $B_1$ is not collapsed (of course if $y_1'$ lies in $e_1 \cap e_3$ we have 2 such subregions, either of which serves our purpose). Furthermore, as $y'_1$ does not lie in $T_{y_2}$, we may shift $\alpha$ over $T_{y_2}$ to realize $B_1$ in $S'$.

\begin{figure}[htb]
\centering
\includegraphics[scale=1.1]{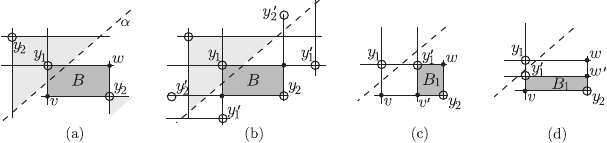}
\caption[]{Figures for Lemma~\ref{lem_upward}, case (2).}
\label{fig_s'_b}
\end{figure}

\end{enumerate}

Finally, let $B_2 \in S'$ be minimal among non-collapsed regions, in the sense that it contains no other. Observe then that no element of $[y_i]_\alpha$, for either $i$, lies in the interior of an edge of $B_2$. Indeed, suppose otherwise, so that e.g. $y'_1 \in [y_1]_\alpha$ is such that $y'_1$ lies along $e_1$ (the other cases are essentially identical). Then $y'_1$ is the endpoint of a segment of $\gamma_1 \cap B$ which cuts across $B_2$, and whose other endpoint lies in $\pos'_1$, thus determining a sub-region. Such a subregion will again be non-collapsed, contradicting minimality.

The region $B_2$ is thus a preimage whose \cpoints map to the given $\hat{y}_1$ and $\hat{y}_2$. By Lemma~\ref{lem_region_image} these are then the \cpoints of an image of $B$ as desired.

\end{proof}

We require a similar statement for the non-upward case: 

\begin{lem}\label{lem_downward}
Let $B \in \R(\pos_1,\pos_2)$ be a non-upward region with \cpoints $y_1$ and $y_2$, $\alpha \in SCC(\pg)$, and let $\hat{y}_i \in i_\aclass (y_i)$ for $i=1,2$. Then $\hat{y}_1$ and $\hat{y}_2$ are the \cpoints of an image of $B$.
\end{lem}

\begin{proof}

Ideally, one would like to find $\alpha' \in [\alpha]$ such that $B$ is a preimage for $i_{\alpha'}$ which maps to the desired region. Supposing for the moment that this is possible, we will break the isotopy $\alpha \leadsto \alpha'$ into a sequence of shifts, as follows: if  $y_i^{\alpha} \neq \hat{y}_i$ for either $i$, there is a unique triangular region $T_i$ such that $y_1^{h_1^{T_i}(\alpha)}=\hat{y}_i$ (Figure~\ref{fig_independent_triangles}(a)). Assuming these shifts may be performed simultaneously, we obtain a curve $\alpha''$ (Figure~\ref{fig_independent_triangles}(b)) which gives the correct \cpoint images. Now, for each \dpoint, let $T_3$ (for $v$) and $T_4$ (for $w$) denote the maximal triangular region contained in $B$ and with vertex the given \dpoint. Isotoping over each of these, we then obtain $\alpha'$ with only vertical intersections with $B$, so $B$ is a preimage with image $B^{\alpha'}$ satisfying the desired properties (Figure~\ref{fig_independent_triangles}(c)).

\begin{figure}[htb]
\centering
\includegraphics[scale=1.4]{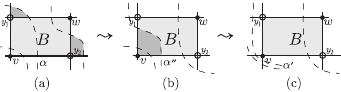}
\caption[]{Shift-isotoping $\alpha$ so that $B$ is the prescribed pre-image.}
\label{fig_independent_triangles}
\end{figure}

Of course these shifts are not in general simultaneously realizable; we may however get around this by passing to the universal cover. To see this,  let $\w{B}$ be a lift of $B$ to the universal cover, and $\w{T_i}$, for each $i$, the lift of $T_i$ (defined in the previous paragraph) with a vertex in common with $\w{B}$. Of course we may always assume $\alpha$ has been isotoped over $T_1$. But then, as $\w{T_2}$ has no interior intersection with $\w{B}$, we may isotope $\rho^{-1}\alpha$ over $\w{T_2}$ without crossing $\w{y_1}$. Similarly, any downward region contained in $\w{B}$ contains neither $\w{y_i}$, while any pair of such regions is either disjoint, or such that one contains the other. In particular, the composition $h_t^{\mathcal{\w{T}}} := h_t^{\mathcal{\w{T}}_4} h_t^{\mathcal{\w{T}}_3}h_t^{\mathcal{\w{T}}_2}$ isotopes $\rho^{-1}\alpha$ to a representative $(\rho^{-1}\alpha)'$ with the desired properties.

\begin{figure}[htb]
\centering
\includegraphics[scale=1.1]{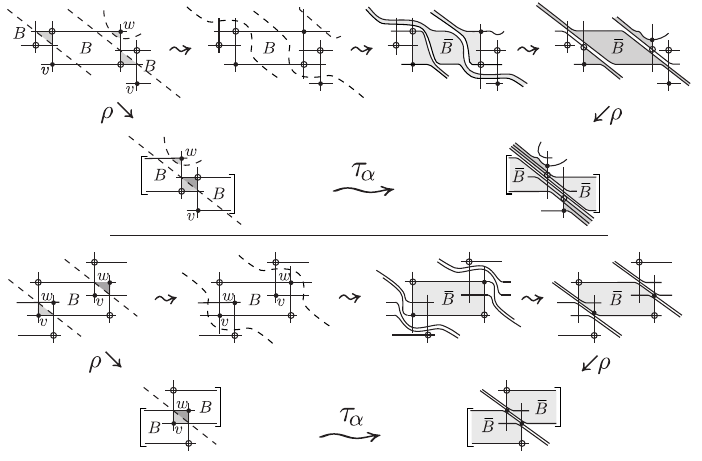}
\caption[]{Examples of the steps outlined in Lemma \ref{lem_downward}. In each of the two figures, the lower left illustration is of some collection of triangular regions in $\pg$ whose shifts cannot be done simultaneously, while each upper row is in the universal cover, and the squiggly arrows refer to, from left to right, the isotopy $(\rho^{-1}\alpha) \leadsto (\rho^{-1}\alpha)'$, the map $D_{h^{\mathcal{\w{T}}}_t(\rho^{-1} \alpha)}$ applied to the the $\rho^{-1}(\mon(\gamma_j))$, and finally the isotopy of these images back to the $\rho^{-1}(\tau_\alpha(\mon(\gamma_j)))$}.
\label{fig_downward}
\end{figure}

We then define $D_{h^{\mathcal{\w{T}}}_t(\rho^{-1} \alpha)}:\w{\pg} \rightarrow \w{\pg}$, as the composition of the $D_{h_t^{\w{T}}\w{\alpha}_j} := h_t^{\w{T}}\w{D_{\alpha_j}}(h_{t}^{\w{T}})^{-1}$ for each connected component $\w{\alpha}_j$ of $\rho^{-1} \alpha$. The map $D_{h^{\mathcal{\w{T}}}_1(\rho^{-1} \alpha)}$ is thus isotopic to the lift $\w{D_\alpha}$, and fixes each $\w{y_i}$ to a lift of $\prey_i$ (recall that intersection points of $\gamma_i \cap \mon(\gamma_j)$ are only considered up to isotopy of $\mon(\gamma_j)$ which involves no bigons). As all intersections with $\w{B}$ are vertical, $\w{B}$ is a preimage, and in particular the $\w{\prey_i}$ bound a region of $\R(\rho^{-1}(\pos^\alpha_1),\rho^{-1}(\pos^\alpha_2))$, which is then mapped by $\rho$ to a region of $\R(\pos^\alpha_1,\pos^\alpha_2)$ as desired. As an illustration, Figure~\ref{fig_downward} goes through the steps for 2 particular configurations.

\end{proof}

We arrive finally at the proof of Theorem~\ref{thm_rp}:

\begin{proof}(of Theorem~\ref{thm_rp})
Let $\gamma_1$ and $\gamma_2$ be disjoint properly embedded arcs in $\pg$, and $\mon \in \mcg$ admit a positive factorization $\fac$. We will show by induction on $n$ that, if $\pos^j_k$ denotes the right position associated to $(\gamma_k,\tau_{\alpha_j} \cdots \tau_{\alpha_1})$ by the algorithm of this section, then $\pos^j_1$ and $\pos^j_2$ are consistent for each $j$. For the base case $n=1$, observe that each region of $\R(\pos^1_1,\pos^1_2)$ is created, so completed by Lemma~\ref{lem_collapsed}. Suppose then that $\pos^{j-1}_1$ and $\pos^{j-1}_2$ are consistent, and let $B^\alpha \in \R(\pos^j_1,\pos^j_2)$. Now, if $B^\alpha$ is created by $\alpha$, then again by Lemma~\ref{lem_collapsed} it is completed, so we assume otherwise. Letting $v^\alpha$ and $w^\alpha$ denote the \dpoints of $B^\alpha$, by definition there are $v \in \pos^{j-1}_1$ and $w \in \pos^{j-1}_2$ such that $v^\alpha \in i_{[\alpha]}(v)$, and $w^\alpha \in i_{[\alpha]}(w)$. Let $\refy_1,\refy_1^\alpha,\refy_2$ and $\refy_2^\alpha$ denote the reflections of $v,w,v^\alpha$ and $w^\alpha$, respectively (Figure~\ref{fig_reflected_regions}(a)). We then have $\refy_i \in i_{[\alpha]}(\refy^\alpha)$, for $i=1,2$ (see Remark~\ref{rem_reflection}). Then, as $\refy_1^\alpha$ and $\refy_2^\alpha$ are the \cpoints of the reflection $\refB^\alpha$ of $B^\alpha$, Lemma~\ref{lem_upward} (if $\refB^\alpha$ is upward with respect to $\alpha$) or \ref{lem_downward} (otherwise) ensure the $\refy_i$ are \cpoints for a region $\refB$. The reflection of $\refB$, which we denote $B$, is thus an element of $\R(\pos^{j-1}_1,\pos^{j-1}_2)$, and such that $B^\alpha$ is an image of $B$ (Figure~\ref{fig_reflected_regions}(b)).

\begin{figure}[htb]
\centering
\includegraphics[scale=1.5]{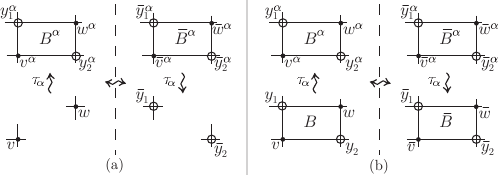}
\caption[]{}
\label{fig_reflected_regions}
\end{figure}

By Lemmas~\ref{lem_upward} and \ref{lem_downward}, any non-collapsed region with the opposite orientation of $B$, and \cpoints in $[y_i]_\alpha$, $i=1,2$, will have an image which completes $B^\alpha$. Consider then the set $S$ of regions of $\R(\pos^{j-1}_1,\pos^{j-1}_2)$ with \cpoints in $[y_i]_\alpha$, $i=1,2$, and the opposite orientation of $B$. We will show that $S$ contains a non-collapsed element. Note firstly that, by consistency, $B$ itself has a completion $B'$ which lies in $S$, so $S$ is non-empty.

Now, we may order $S$ by considering \cpoints along $\gamma_1$ in accordance with its standard orientation (i.e. pointing from $y_1$ to $v$); let $B'_m$ denote a minimal element. Suppose then that $B'_m$ is collapsed. Then by Lemma~\ref{lem_collapsed} there is a region $B_{m+1}$  of opposite orientation, and the same \dpoint set, as $B'_m$, and whose \cpoints $y_i^{m+1}$ lie in $[y_i]_\alpha$ (Figure~\ref{fig_maximal_region}(b)). In particular, $B_{m+1}$ lies in $\R(\pos^{j-1}_1,\pos^{j-1}_2)$, and so is completed by some $B'_{m+1}$ which lies in $S$, but is such that its \cpoint along $\gamma_1$ is further along $\gamma_1$ than that of $B'_m$, contradicting minimality. Thus $B'_m$ is non-collapsed, and so has an image which completes $B^\alpha$.

\begin{figure}[htb]
\centering
\includegraphics[scale=1.5]{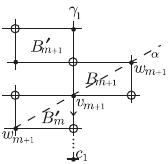}
\caption[]{}
\label{fig_maximal_region}
\end{figure}

\end{proof}

\section{Restrictions on $p.e.(\mon)$}\label{sec_d}

In this section, we switch focus back to the entirety of a pair of
distinct properly embedded arcs $\gamma_1, \gamma_2$ in a surface $\pg$, and the images of these arcs
under a right-veering diffeomorphism $\mon$. The motivating
observation here is that, as the endpoints of each arc are by definition included in any right position, the property of the initial horizontal segments of the images determining an initially parallel region is independent of right position, and so is a property of the pair $\mon(\gamma_1),\mon(\gamma_2)$. In particular, for \emph{positive} $\mon$, if $\mon(\gamma_1)$ and $\mon(\gamma_2)$ are initially parallel, they must admit consistent right positions $\pos_i(\mon,\gamma_i)$ in which the initially parallel region is completed (see Example \ref{ex_rp}). We are interested in
understanding what necessary conditions on $\alpha \in p.e.(\mon)$ (Definition~\ref{def_pe})
can be derived from the information that
$\mon(\gamma_1)$ and $\mon(\gamma_2)$ are initially parallel.

\subsection{Restrictions from initial horizontal segments}

Throughout the section, we will be considering rectangular
regions $R$ in $\pg$, and classifying various curves and
arcs by their intersection with such regions. We need:

\begin{defin}\label{def_arcs}
Let $R$ be an embedded rectangular region with distinguished oriented edge
$e_1$, which we call the \emph{base} of $R$. We label the remaining
edges $e_2,e_3,e_4$ in order, using the orientation on $e_1$ (Figure
\ref{fig_arcs}). We classify properly embedded (unoriented) arcs on
$R$ by the indices of the edges corresponding to their boundary
points, so that an arc with endpoints on $e_i$ and $e_j$ is of \emph{type}
$[i,j]$ (on $R$). We are only concerned with arcs whose endpoints are not on
a single edge. An arc on $R$ is then

\begin{itemize}

\item \emph{horizontal} if of type $[2,4]$

\item \emph{vertical} if of type $[1,3]$

\item \emph{upward} if of type $[1,2]$ or $[3,4]$

\item \emph{downward} if of type $[1,4]$ or $[2,3]$

\item \emph{non-diagonal} if horizontal or vertical

\end{itemize}

\end{defin}

\begin{figure}[h!]
\centering \scalebox{.6}{\includegraphics{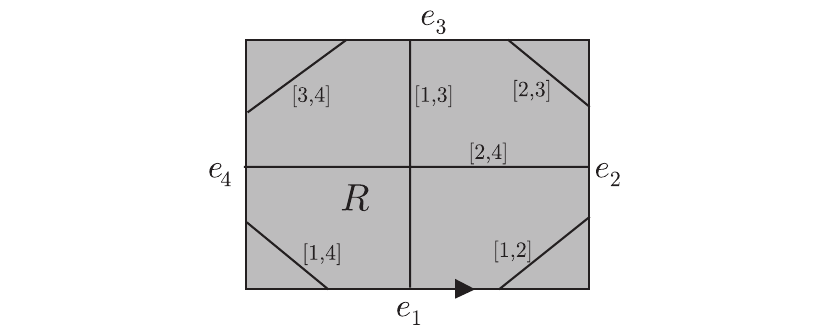}}
\caption[Arc types]{Representatives of each possible type of arc on
$R$} \label{fig_arcs}
\end{figure}

\begin{defin}\label{def_type_1}
Let $R$ in $\pg$ be an rectangular region with
distinguished base as in Definition~\ref{def_arcs}. We say $\alpha
\in SCC(\pg)$ is \emph{not downward on R} if each arc $\alpha \cap R$ is
not downward on $R$.
\end{defin}

\begin{remark}\label{D}
Following Definition~\ref{def_arcs}, we use a given pair of properly
embedded arcs to define a rectangular region $D$ of $\pg$
as follows: Suppose  $\gamma_1$ and $\gamma_2$ are disjoint properly embedded arcs,
where $\partial{\gamma_i} = \{c_i,c_i'\}$, and $\bar{\gamma_1}$ is a
given properly embedded arc with endpoints $c_1'$ and $c_2$. Let
$\bar{\gamma_2}$ be a parallel copy of $\bar{\gamma_1}$ with
endpoints isotoped along $\gamma_1$ and $\gamma_2$ to $c_1$ and $c'_2$. We then define
$D$ to be the rectangular region bounded by
$\gamma_1,\gamma_2,\bar{\gamma_1}, \bar{\gamma_2}$ (the endpoints
are labeled so that $c_1$ and $c_2$ are diagonally opposite), and
base $\bar{\gamma_2}$, oriented away from $c_1$ (see Figure
\ref{fig_disc}). For the remainder of this section, $D$ will always refer to this construction. Of course, for a given pair of arcs, $D$ is unique only up to the choice of $\bar{\gamma_1}$.

\end{remark}
\begin{figure}[h]
\centering \scalebox{.6}{\includegraphics{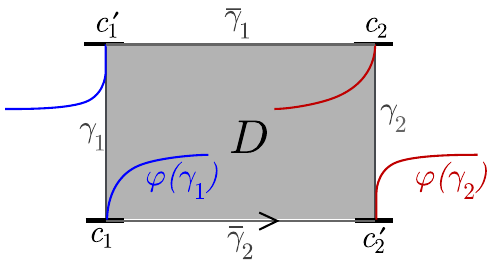}} \caption[Disc construction]{disc
construction} \label{fig_disc}
\end{figure}

\begin{defin}\label{def_flat}
We say the pair $\mon(\gamma_1), \mon(\gamma_2)$ is \emph{initially parallel (on
$D$)} if there exists $D$ such that $\mon(\gamma_1),\mon(\gamma_2), \gamma_1,\gamma_2$ bound a rectangular region $B \hookrightarrow D$ on which $c_1$ and $c_2$ are vertices. An initially parallel (on $D$) pair of images is \emph{flat} (on $D$) if $\mon(\gamma_i) \cap \bar{\gamma_j}
= \emptyset$ for $ i,j\in\{1,2\}$.
\end{defin}

As expected, then, the pair of arcs $\mon(\gamma_1), \mon(\gamma_2)$ is initially parallel exactly when the horizontal segments $h_{c_1}$ and $h_{c_2}$ originating from the boundary of each pair of right positions determine an initially parallel region.
Using this, we immediately obtain:

\begin{lem}\label{lem_not_downward}

Let $\mon$ be positive, $D$ as in Remark \ref{D} and the pair $\mon(\gamma_1), \mon(\gamma_2)$ initially parallel on $D$. Then $\alpha \in p.e.(\mon)$ only if
$\alpha$ is not downward on $D$.

\end{lem}

\begin{proof}
Suppose otherwise - then $\alpha \cap D$ has an arc of type $[1,4]$
or $[2,3]$ on $D$. Now, if $\alpha \in p.e.(\mon)$, there is some
positive factorization of $\mon$ in which $\tau_\alpha$ is the initial Dehn
twist. However, if $\alpha \cap D$ has an arc of type $[1,4]$ ($[2,3]$), then the initial segment of $\tau_\alpha(\gamma_1)$ (resp. $\tau_\alpha(\gamma_2)$) is strictly to the right of $\mon(\gamma_1)$ (resp. $\mon(\gamma_2)$), a contradiction.
\end{proof}

\subsubsection{Motivating examples}
We begin with a pair of examples, meant to motivate the various definitions to come.
\begin{example}\label{example_positions}

Consider a positive $\mon \in \mcg$, and a pair of properly embedded arcs $\gamma_1$ and $\gamma_2$ such that the images $\mon(\gamma_1)$ and $\mon(\gamma_2)$ are flat in some disc $D$ (Figure \ref{fig_minimal_position_pl}(a)). The segments $h_{c_1}$ and $h_{c_2}$ are of course initially parallel (along some region $B$), so completed by some $v\in \pos_1$ and $w \in \pos_2$ (it is easiest to keep in mind the simplest case in which $v$ and $w$ are the opposite endpoints of $\gamma_1$ and $\gamma_2$, as in the figure). Let $\alpha$ be a simple closed curve such that  $\alpha \cap D$ has exactly two diagonal arcs, each upward, and connected such that an orientation of $\alpha$ gives each of these arcs the same orientation. The arcs $\tau_\alpha^{-1}(\mon(\gamma_i))$ are thus again initially parallel, with \cpoints which map to those of $B$ under $i_\alpha$. One may then verify that there is a unique completing region, which is collapsed by $\alpha$. Lemma~\ref{lem_collapsed} then guarantees that, if $v_2$ and $w_2$ are the \dpoints of this completing region, these points define initially parallel segments $h_{v_2}$ and $h_{w_2}$. This new initially parallel region is then completed by (an image under $i_\alpha^{-1}$ of) the original completing region for $B$. Furthermore, the \cpoints of the regions are each fixable. The regions and points involved are illustrated in Figure \ref{fig_minimal_position_pl}(b).

Consider then a slightly more complicated configuration, in which $\alpha$ has 4 diagonal arcs, arranged as in Figure \ref{fig_minimal_position_pl}(c). Again, one may verify that any consistent right positions for $((\tau_\alpha^{-1} \circ \mon),\gamma_i)$ (Figure \ref{fig_minimal_position_pl}(d)) must contain one of the two positions indicated in Figure \ref{fig_minimal_position_pl}(e) and (f). Again, each of these determines a sequence of (collapsed) regions and points which ends with the original completions. The goal of the remainder of this section is to show that for \emph{any} simple closed curve $\alpha$, any pair of consistent right positions for $((\tau_\alpha^{-1} \circ \mon),\gamma_i)$ must contain a similar sequence, which then determines necessary conditions on $\alpha \in p.e. (\mon)$.
\end{example}

\begin{figure}[htb]
\centering
\includegraphics[scale=.7]{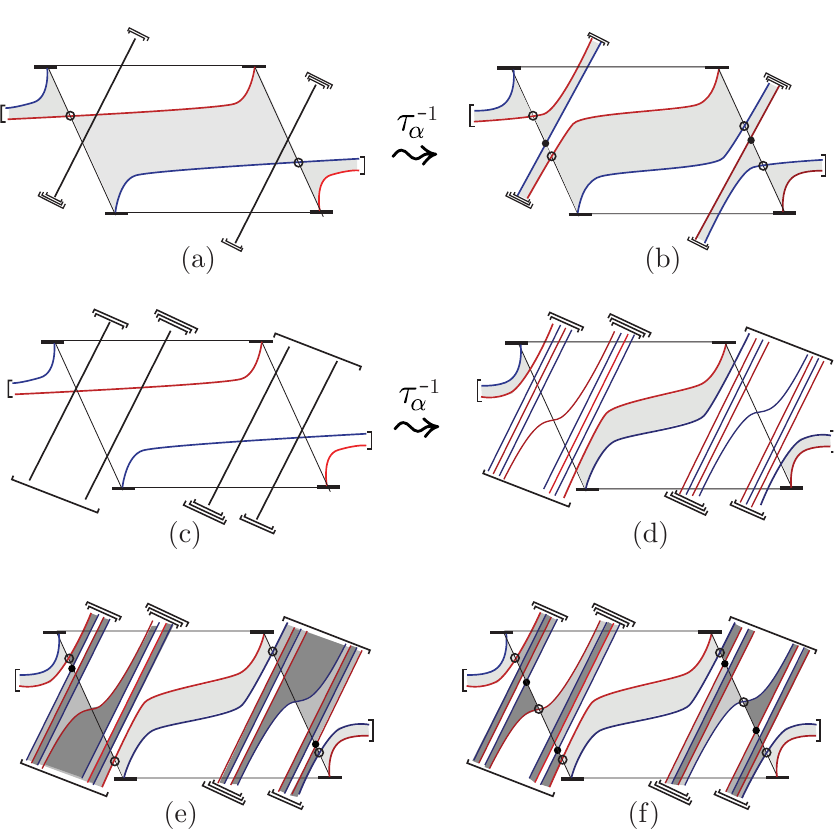}
\caption{  }
\label{fig_minimal_position_pl}
\end{figure}

\subsubsection{Notation and conventions}
Given a pair of arc images $\mon(\gamma_1)$ and $\mon(\gamma_2)$ flat on some $D$, and a non-downward curve $\alpha$, let $\preB \hookrightarrow D$ denote the initially parallel region with \dpoints $c_1$ and $c_2$, and $\prey_i \in \gamma_i$, $i=1,2$, the \cpoints. Consider the set of upward triangles of $(\alpha,\gamma_1,\mon(\gamma_2))$ with vertex $\prey_1$ (Figure \ref{fig_alpha_cap_D} (a)); each of course has a distinct vertex in $\gamma_1 \cap \alpha$, which we label $x_1,x_3,\ldots,x_{m_1}$ such that indices decrease along $\gamma_i$ in the direction of the orientation of $\gamma_1$. Similarly we have upward triangles with vertex $\prey_2$, and vertices $x_2,x_4,\ldots,x_{m_2}$ along $\gamma_2$ (Figure~\ref{fig_alpha_cap_D} (b)). Let $a_1, a_2 ,\dots a_m$, where $a_1=1$, be the indices of the $x_i$ with order given by the order in which they are encountered travelling along $\alpha$ to the right from $x_1$. We associate to $\alpha$ the list $\pi_{\alpha} = (a_1, a_2 ,\dots a_m)$, which for reasons which will become apparent in the following subsection we define only up to cyclic permutation. We decorate an entry with a bar, $\overline{a_j}$, if, for odd (even) $j$, the orientation of the intersection point $x_j$ agrees (disagrees) with that of $x_1$.

\begin{figure}[h!]
\centering \scalebox{.8}{\includegraphics{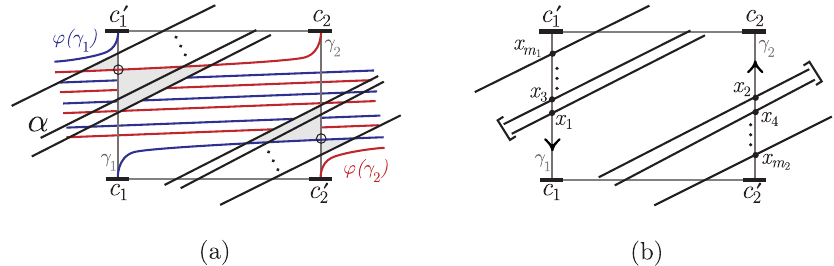}}
\caption[Arcs of $\alpha \cap D$]{(a) Triangular regions on the \cpoints of $B^\alpha$. (b) Labeling of arcs and intersections. The paths $\eta_{2,3}$ and $\eta_{4,1}$ are components of a symmetric multipath.}
\label{fig_alpha_cap_D}
\end{figure}

Given a pair $x_i \in \alpha \cap \gamma_2$ and $x_j \in \alpha \cap \gamma_1$ (so $i$ is even, $j$ odd), the pair splits $\alpha$ into two disjoint arcs. If the orientations of $x_i$ and $x_j$ disagree, we label these arcs $\eta_{i,j}$ and $\eta_{j,i}$ (where for even $i$, $\eta_{i,j}$ refers to the arc connecting $x_i$ and $x_j$ and initially \emph{exterior} to $D$ at each endpoint, so for odd $j$, a neighborhood of each endpoint of $\eta_{j,i}$ is contained in $D$). We refer to each such arc as a \emph{path} of the pair $(\alpha,D)$. We say paths $\eta$ and $\eta'$ are \emph{parallel} if, after sliding endpoints along the $\gamma_i$ so as to coincide, they are isotopic as arcs (relative to their endpoints). Then a \emph{multipath} of $(\alpha,D)$ is a collection of more than one pairwise parallel path.

Of particular interest will be the following special case:

\begin{defin}\label{def_multipath}
A multipath as described in the previous paragraph is \emph{symmetric} if its components are $\eta_{a_1,b_1},\eta_{a_2,b_2},\ldots , \eta_{a_r,b_r}$, where each pair $\{a_i,b_{r+1-i}\}$ is $\{j,j-1\}$ for some even $j$. So, for example, any symmetric \emph{path} is $\eta_{j,j-1}$ or $\eta_{j-1,j}$ for some even $j$.  Figure~\ref{fig_alpha_cap_D} (b) contains an example of a symmetric multipath with components $\eta_{2,3}$ and $\eta_{4,1}$.
\end{defin}

\subsection{Necessary conditions for $\alpha \in p.e.(\mon)$}

This subsection is devoted to introducing the terminology of, and proving, Theorem~\ref{thm_nested}.

\begin{defin}\label{def_nested}

Let $\alpha \in SCC(\pg)$ be not downward on a region $D$. We call $\alpha$ \emph{balanced} (with respect to $D$) if each $x_i$ is an endpoint in either a symmetric path or a component of a symmetric multipath. We say balanced $\alpha$ is \emph{nested} with respect to $D$ if its associated $\pi_{\alpha}$ can be reduced to the empty word by successive removals of pairs $(i,j)$ of consecutive entries, where $\eta_{i,j}$ is a symmetric path or component of a symmetric multipath. 
\end{defin}

Note that, in the absence of symmetric multipaths of more than one component, these properties are encoded in $\pi_\alpha$. For example, $(12\overline{3}4)$ is not balanced, $(135462)$ is balanced but not nested, and $(1\overline{4}56\overline{3}2)$ is nested. Also note that each condition requires the number of intersections $|\alpha \cap \gamma_1|$ to equal $|\alpha \cap \gamma_2|$; i.e. $m_1 + 1 = m_2$.

We want to characterize the nested condition in terms of paths of $\alpha$. We need:

\begin{defin}
Let $\eta_{a,b},\eta_{c,d} \subset \alpha$ be paths, with $a,b,c$ and $d$ distinct. We refer to the pair as \emph{nested} if either one contains the other, or they are disjoint. Otherwise, the pair is \emph{non-nested}.
\end{defin}

Observe then that, if for some balanced $\alpha$, each pair of components of symmetric multipaths is nested, then there is some `innermost' multipath whose paths contain no other. In particular the indices of the endpoints of these paths are consecutive pairs in $\pi_\alpha$. We have:

\begin{lem}
Let balanced $\alpha$ be such that each pair of components of symmetric multipaths is nested. Then $\alpha$ is nested.
\end{lem} 

\begin{flushright}
$\square$ 
\end{flushright}

To keep track of intersection points $\gamma_i \cap \tau_\alpha^{-1}(\mon(\gamma_j))$, $i,j \in \{1,2\}$, consider the restriction of $\tau_\alpha^{-1}$ to a neighborhood $U(\alpha)$ of $\alpha$ (Figure~\ref{fig_alpha_inverse_psi}). Observe that each \cpoint of the initially parallel region with \dpoints $c_1$ and $c_2$ is fixable under $\tau^{-1}_\alpha$. As in the previous subsection we label these $\prey_1 \in \gamma_1 \cap \mon(\gamma_2)$ and $\prey_2  \in \gamma_2 \cap \mon(\gamma_1)$. There are then exactly $(m_1+3)/2$ points to which $\prey_1$ may be fixed by $i_\alpha^{-1}$; we label these $y_1,y_3,\ldots,y_{m_1+1}$ such that the indices increase along $\gamma_1$ from $c_1$. Similarly, there are $(m_2+2)/2$ possible fixed preimages of $y_2$, which we label $y_2,y_4,\ldots,y_{m_2+2}$ along $\gamma_2$ from $c_2$.

\begin{figure}[h!]
\centering \scalebox{.9}{\includegraphics{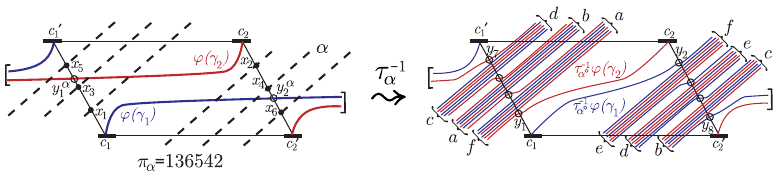}}
\caption[Example of inverse image of $D$]{Above left, the various points of $\alpha \cap \gamma_i$ and $\alpha \cap \mon(\gamma_i)$. Above right, the inverse image $\tau_\alpha^{-1} (\mon (\gamma_i))$, for $\alpha$ with
$\pi_\alpha = 136542$ (strands terminating in brackets with like letters are meant to be identified).} \label{fig_alpha_inverse_psi}
\end{figure}

\subsubsection{The right positions $\pos^*_1$ and $\pos^*_2$}

Our immediate goal is to show that, if $\tau_{\alpha}^{-1}(\mon(\gamma_i)), i=1,2$ admit consistent right positions, then any such pair of right positions has a particular  pair of sub-positions whose regions extend non-trivially along the entire length of each arc, as described in Example~\ref{example_positions}. Theorem~\ref{thm_nested} will then follow from various properties of these right positions.

\begin{lem}\label{lem_index_chain}
Let $\mon(\gamma_1)$ and $\mon(\gamma_2)$ be flat with respect to some $D$, $\alpha \in p.e.(\mon)$, and  $\pos_i, i=1,2$ consistent right positions of $(\tau_{\alpha}^{-1} \circ \mon, \gamma_i)$. Then there are $\pos^*_1 = \{v_j\}_{j=1}^n \subset \pos_1$, $\pos_2^* = \{w_j\}_{j=1}^n \subset \pos_2$, and $\{e_{k}\}_{k=1}^{2n} \subset \mathbb{N}$ such that

\begin{enumerate}

\item   Each of the sets $\{v_j\}$ and $\{w_j\}$ is a right position (i.e. $v_1 = c_1$ and $v_n = c'_1$, and $w_1=c_2$ and $w_n=c'_2$), and the indices $1,\cdots,n$ increase along each of $\gamma_i$ and $\tau^{-1}_\alpha(\mon(\gamma_i))$ from the endpoint $c_i$.

\item For each $j < n$, segments $h_{v_j,v_{j+1}}$ and $h_{w_j,w_{j+1}}$ are initially parallel along a region $B_j$, and completed along a region $B'_j$ with \dpoints $v_{j+1}$ and $w_{j+1}$. Furthermore, the set of \cpoints for these regions are fixable, and the \cpoints of the completing pair $B_j, B'_j$ are $y_{e_{2j-1}}$ and $y_{e_{2j}}$ (in the set $\{y_k\}$ of fixable points described above). 

\end{enumerate}
See Figure~\ref{fig_minimal_regions} for a schematic illustration.
\end{lem}

\begin{figure}[htb]
\centering
\includegraphics[scale=.8]{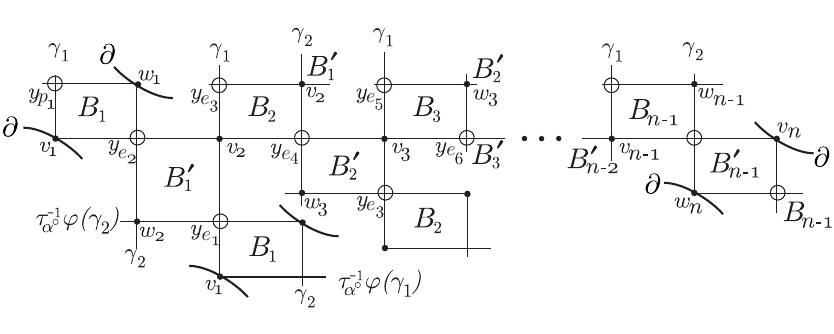}
\caption{A schematic overview of the positions $\pos_1^*$ and $\pos_2^*$ constructed in Lemma~\ref{lem_index_chain}. The picture is of the universal cover of $\pg$, along a lift of $\tau^{-1}_{\alpha}(\mon(\gamma_1))$. As usual, all lifts of an object are given the same label as the original object.}
\label{fig_minimal_regions}
\end{figure}

\begin{proof}
 By assumption, $h_{c_1}$ and $h_{c_2}$ are initially parallel along a region $B_1$, with fixable \cpoints $y_1$ and $y_2$, which for consistency must be completed by some region $B'_1$ with \dpoints which we label $v_2$ and $w_2$.  We need then to show that the horizontal segments $h_{v_2,c'_1}$ and $h_{w_2,c'_2}$ are themselves initially parallel along a region with \cpoints again in the fixable point set $\{y_k\}$. The result will then follow, as we may repeat the construction until we hit one of the opposite endpoints.

Now, each of the sets $\{y_1,y_3, \ldots ,y_{m_1+2}\}$ and $\{y_2,y_4,\ldots,y_{m_2+2}\}$ is of course just the $\alpha-$equivalence class (Definition~\ref{def_alpha_equiv}) of each of its elements. In particular, each edge of $B'_1$ is contained in the span of the $\alpha$-equivalence class of its \cpoint; as $\alpha$ may be assumed to intersect $B'_1$ non-trivially, $B_1'$ is collapsed by $\alpha$, and so by Lemma~\ref{lem_collapsed}, there is a complementary $(B_1')^c$, with \cpoints in $[y_i]_{\alpha}$, as desired.

Continuing in this fashion, so for each $j$, $B_j = (B'_{j-1})^c$, we obtain a sequence of points $\{v_j\}$ and $\{w_j\}$ satisfying condition (2), which continues until we hit an endpoint; i.e. (assuming $m_1 \geq m_2-1$), $w_n = c'_2$. Condition (1) will then be satisfied if $m_1 = m_2-1$, so that the corresponding $v_n$ is the other endpoint $c'_1$.

 Suppose otherwise; we will show that $B_{n-1}$ has no completion, a contradiction. As elsewhere in the paper, we will use the notation $[a,b]_\sigma$ to denote the segment of an arc $\sigma$ between points $a$ and $b$. Let ${\preB}'$ denote the completion of $\preB$, and consider firstly the case that $\alpha \cap {\preB}'$ has no vertical arcs (i.e. arcs connecting $\mon(\gamma_1)$ and $\mon(\gamma_2)$ - Figure~\ref{fig_no_path_general} (a)). Then, by construction, $y_{2n-2}$ is such that $e := [y_{2n-2},c'_1]_{\tau_\alpha^{-1}(\mon(\gamma_1))}$ has no interior intersection with $[y_1,c'_1]_{\gamma_1}$ (Figure~\ref{fig_no_path_general} (b)). As such, any completion of $B_{n-1}$ has $e$ as an edge, and $c'_1$ as a \dpoint. Thus $[y_{2n-3},c'_1]_{\gamma_1}$ is another edge. However, it is clear that e.g. the intersection number of this edge with $\tau_\alpha^{-1}(\mon(\gamma_1))$ is greater than that of any candidate edge along $\gamma_2$, so we have no possible region. The remaining case, that $\alpha \cap {\preB}'$ has vertical arcs, is almost identical: we now take $e$ to be $[y_{2n-2},c'_2]_{\gamma_2}$, which now has no interior intersection with either $\tau_\alpha^{-1}(\mon(\gamma_i))$, so is an edge of any completion. Thus $[y_{2n-3},c'_2]_{\tau_\alpha^{-1}(\mon(\gamma_2))}$ is another edge, and again considering intersections of $\tau_\alpha^{-1}(\mon(\gamma_1))$ with the edges along the $\gamma_i$ we conclude that no region exists. 
 
\end{proof}

\begin{figure}[h!]
\centering \scalebox{.6}{\includegraphics{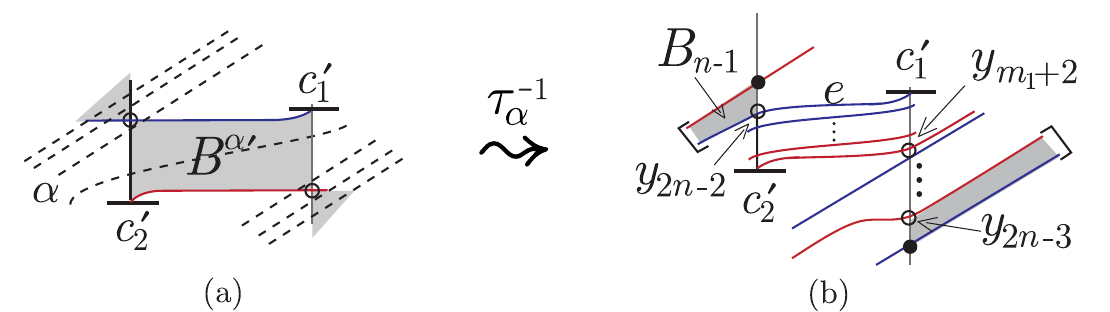}}
\caption[Discs]{}
\label{fig_no_path_general}
\end{figure}

We have the following immediate corollary, giving us `half' of Theorem~\ref{thm_nested}:

\begin{cor}\label{lem_balanced}

Let $\mon(\gamma_1)$ and $\mon(\gamma_2)$ be flat with respect to some $D$. Then $\alpha \in p.e.(\mon)$ only if $\alpha$ is balanced with respect to $D$.

\end{cor}

\begin{flushright}
\qed
\end{flushright}

\subsubsection{Nestedness}

The proof of the remainder of Theorem~\ref{thm_nested}, that \emph{nestedness} is a further necessary condition for consistency, in fact also follows easily from the previous lemma, but keeping track of everything involved requires quite a bit of notation. We begin by using Lemma~\ref{lem_index_chain} to characterize nestedness in terms of regions and \dpoints.

As in the previous section, we denote by  $\mathcal{R}(\pos_1,\pos_2)$ the set of initially parallel / completed regions in a pair of right positions $\pos_1$ and $\pos_2$. Recall (Lemma~\ref{lem_collapsed}), that a collapsed region of length $l>1$ and width $w=1$ (resp. length $l=1$ and width $w>1$) contains $l-1$ (resp. $w-1$) collapsed subregions each of $l=w=1$.

\begin{defin}\label{def_corresponding_discs}
Given a pair of complementary paths $\eta_{a,b}$ and $\eta_{b,a}$, where $a$ is \emph{even}, let $B_{a,b}$ and $B_{b,a}$ be the collapsed pair of rectangular regions bounded by $[y_a,y_{a+2}]_{\gamma_2}$, $[y_a,y_{a+2}]_{\tau_\alpha^{-1}(\mon(\gamma_2))}$, $[y_b,y_{b+2}]_{\gamma_1}$, and $[y_b,y_{b+2}]_{\tau_\alpha^{-1}(\mon(\gamma_1))}$, where  $B_{a,b}$ has vertices $y_a$ and $y_b$, and $B_{b,a}$ has vertices $y_{a+2}$ and $y_{b+2}$. (Figure~\ref{fig_corresponding_regions} (a)). Thus $B_{a,b}$ and $B_{b,a}$ are collapsed regions satisfying length=width=1. We refer to each such region as \emph{basic}. We refer to a pair of basic regions $B_{a,b}$ and $B_{c,d}$ as \emph{nested} or \emph{non-nested} if the associated paths $\eta_{a,b}$ and $\eta_{c,d}$ are nested or non-nested.  From the definitions, we see that, if $B_{a,b}$ and $B_{c,d}$ are nested, they are either disjoint, or intersect such that no vertex is in the intersection; i.e. in `strips' as in Figure~\ref{fig_corresponding_regions} (b). Conversely, if the regions are non-nested, they intersect in a `corner overlap' as in Figure~\ref{fig_corresponding_regions} (c). Finally, we call a region $B \in \mathcal{R}(\pos^*_1,\pos^*_2)$ \emph{nested} if each basic subregion of $B$ is nested with each basic subregion of all other regions of $\mathcal{R}(\pos^*_1,\pos^*_2)$.
\end{defin}

\begin{figure}[htb]
\centering
\includegraphics[scale=.8]{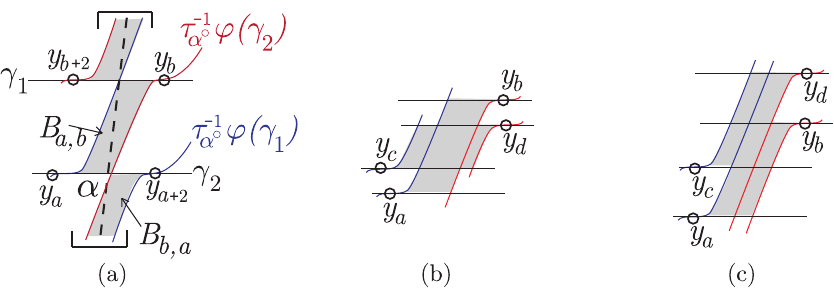}
\caption{(a) The region $B_{a,b}$ corresponding to the path $\eta_{a,b}$, and complementary $B_{b,a}$ corresponding to $\eta_{b,a}$. The curve $\alpha$ is the dashed line. (b) Nested, and (c) non-nested regions.}
\label{fig_corresponding_regions}
\end{figure}

We can thus characterize nestedness in terms of regions as follows:
\begin{obs}\label{nestedness_characterization}
A balanced curve $\alpha$ is nested if, for any two regions $A,B \in \mathcal{R}(\pos^*_1,\pos^*_2)$, any pair consisting of a basic region of $A$ and a basic region of $B$ is nested. Equivalently, if each $B \in \mathcal{R}(\pos^*_1,\pos^*_2)$ is nested.
\end{obs}

We need:

\begin{defin} Given right positions $\pos_1, \pos_2, \pos_3$ and $\pos_4$, we call a region $B \in \mathcal{R}(\pos_1,\pos_2)$ \emph{empty} (in $ (\pos_3,\pos_4)$) if neither the interior of $B$ nor the interior of the edges which form $\partial B$ contain any points of either $\pos_3$ or $\pos_4$. Conversely, any point $y$ which is a \cpoint of some region in $\mathcal{R}(\pos_1,\pos_2)$  is \emph{isolated} (in $\mathcal{R}(\pos_3,\pos_4)$) if, for each $B \in \mathcal{R}(\pos_3,\pos_4)$, $y$ is in neither the interior of $B$ nor the interior of the edges which form $\partial B$.
\end{defin}

Then:

\begin{lem}\label{lem_structure}
Let $\pos_1^*$ and $\pos_2^*$ be the right positions constructed in Lemma~\ref{lem_index_chain}. Then if $B \in \mathcal{R}(\pos^*_1,\pos^*_2)$ is empty (in $ (\pos^*_1,\pos^*_2)$) with isolated \cpoints, then $B$ is nested.

\end{lem}

\begin{proof}
Suppose otherwise, so $B$ is empty with isolated \cpoints, and there is $A \in \mathcal{R}(\pos^*_1,\pos^*_2)$ such that some basic region $B_{a,b}$ of $B$ and some basic region $B_{c,d}$ of $A$ are non-nested. Now, as in the proof of Lemma~\ref{lem_collapsed}, any collapsed region has either width or length equal to 1; the cases are essentially identical, so we assume $l(B) = 1$. Now, if $l(A) \neq 1$, note that $B_{a,b}$ and $B_{d,c} = (B_{c,d})^c$ are also non-nested (Figure~\ref{fig_structure}(a)), so we may replace $A$ with its complementary region $A^c$, which then satisfies $l(A^c)=1$. It is then trivial to verify that, if two collapsed regions each of length one (or each of width one) have non-nested basic regions, then each of the pair either contains a \dpoint of, or has a \cpoint contained in, the other (Figure~\ref{fig_structure}(b)). We therefore arrive at a contradiction.
\end{proof}

\begin{figure}[htb]
\centering
\includegraphics[scale=2.5]{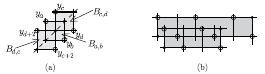}
\caption{(a) If a basic region is non-nested, so is its complementary region. (b) 2 collapsed, non-nested regions of length 1. Each is either non-empty, or has non-isolated \cpoints.}
\label{fig_structure}
\end{figure}

\begin{lem}\label{lem_structure_2}
Let $\pos_1^*$ and $\pos_2^*$ be as in Lemma~\ref{lem_index_chain}. Then if $B \in \mathcal{R}(\pos^*_1,\pos^*_2)$ is empty (in $ (\pos^*_1,\pos^*_2)$) with isolated \cpoints, then the same holds for $B^c$.
\end{lem}

\begin{proof}

Suppose that $B^c$ is not empty, so contains some \dpoint $v$ of some region $A \in \mathcal{R}(\pos^*_1,\pos^*_2)$. Now, the complement in $B$ of the basic regions of $B$ is the same as the complement in $B^c$ of the basic regions of $B^c$, so, as $B$ is empty, $v$ lies in a basic region of $B^c$ (Figure~\ref{fig_structure_2}(a)). But then the basic subregion of $A$ with vertex $v$ is non-nested with a basic subregion of $B$, contradicting the previous lemma. 

\begin{figure}[htb]
\centering
\includegraphics[scale=2]{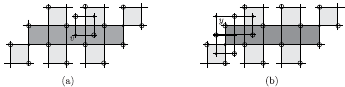}
\caption{Figures for Lemma~\ref{lem_structure_2}. The region $B^c$ is darkly shaded, the basic regions of $B$ lightly. }
\label{fig_structure_2}
\end{figure}

Similarly, if a \cpoint $y$ of $B^c$ is not isolated, it is contained in some $A$ (Figure~\ref{fig_structure_2}(b)). Then, as by construction no basic region contains the \cpoint of any other, $y$ must lie in the complement of the basic regions of $A$; i.e. in $A \cap A^c$. But then $A$ and $B$ again contain non-nested basic regions, a contradiction. 
\end{proof}

Now, to connect $\R(\pos_1^*,\pos_2^*)$ with arbitrary consistent $\R(\pos_1,\pos_2)$, we require:

\begin{lem}\label{lem_empty} Let $\gamma_i,\mon,\pos_i,\pos^*_i$, $i=1,2$, be as in Lemma~\ref{lem_index_chain}. Suppose region $B_j \in \R(\pos^*_1,\pos^*_2)$ is empty in $(\pos_1,\pos_2)$, and the \cpoints $y_{e_{2j-1}}$ and $y_{e_{2j}}$ isolated in $\R(\pos_1,\pos_2)$. Then the completing region $B'_j$ is also empty in $(\pos_1,\pos_2)$.
\end{lem}

\begin{proof} Suppose that $B_j$ is empty, while $B'_j$ is not, so there is some $v \in \pos_i$, $i \in \{1,2\}$, such that $v \in B'_j$. Then $v$ is the \dpoint of an initially parallel region $B \subset B'_j$ whose remaining \dpoint is a \dpoint of $B'_j$ (Figure~\ref{fig_empty} - we have illustrated the case $i=1$). As $B$ is a region of $\R(\pos_1,\pos_2)$, it has a completion $B'$. Then, letting $v'$ be the \dpoint of $B'$ along $\gamma_1$, we see that, as the boundary of $B'$ is assumed to contain neither of the pair $y_{e_{2j-1}}, y_{e_{2j}}$, then $v'$ is in the interior of two distinct edges of $B_j'$. We thus have a contradiction, as each region of $\R(\pos^*_1,\pos^*_2)$ is (by construction) embedded. 
\end{proof}

\begin{figure}[htb]
\centering
\includegraphics[scale=.5]{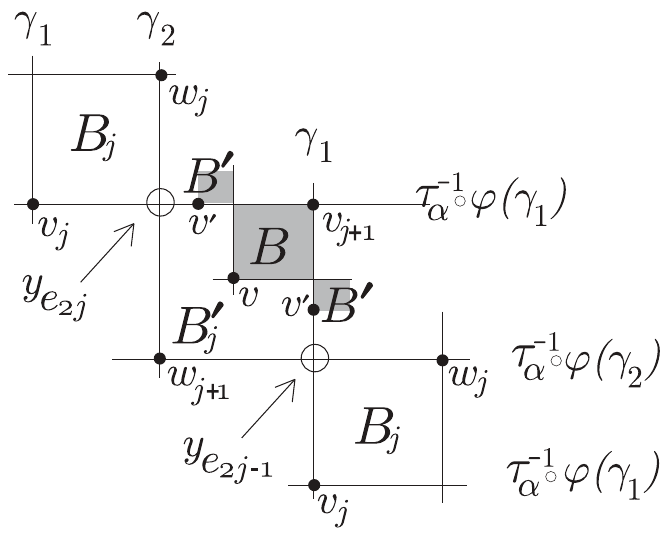}
\caption{Figure for Lemma~\ref{lem_empty}}
\label{fig_empty}
\end{figure}

Observe then that, referring back to the sequence of regions created in Lemma~\ref{lem_index_chain}, if $B_1$ is empty, Lemma~\ref{lem_empty} implies that $B'_1$ is empty, and so by Observation~\ref{nestedness_characterization} it follows that the collection of entries of $\pi_\alpha$ up to the largest index of a basic region of $B'_1$ (i.e. $\{1,2,\ldots, e_3-1\}$) appear in $\pi_\alpha$ as a nested (not necessarily contiguous) subword. Our strategy is then to proceed by modifying the surface $\pg$ via a sequence of punctures, which will have the effect of removing all regions $\mathcal{R}(\pos_1,\pos_2) \setminus \mathcal{R}(\pos^*_1,\pos^*_2)$, and isolating each of the fixed points $y_{e_j}$. This will allow us to generalize Lemma~\ref{lem_structure} to apply to each region of $\mathcal{R}(\pos^*_1,\pos^*_2)$, and then show that the `overlapping regions' which characterize the non-nested case are an obstruction to consistency.

We begin by defining our sequence of modifications. Given $\pos^*_i$ as in Lemma~\ref{lem_index_chain}, let $\pg^j$ be the surface given by deleting a pair of points from a neighborhood of each of the \cpoints $\{ y_{e_k} \ | \ e_k \leq 2j  \}$ as in Figure~\ref{fig_deletion_points}. To be precise, we chose a neighborhood about each of these points such that the restriction of $\{\gamma_1,\gamma_2,\tau_\alpha^{-1}(\mon(\gamma_1)),\tau_\alpha^{-1}(\mon(\gamma_2))\}$ to this neighborhood is a pair of arcs with a single intersection point, so these arcs divide the neighborhood into 4 `quadrants', two of which are filled by the pair of regions for which the given \cpoint is a vertex. We then delete a neighborhood of a point in the interior of each of the remaining two quadrants (Figure~\ref{fig_deletion_points}). Note that this has the effect of removing from $\mathcal{R}(\pos_1,\pos_2)$ each region which contains either of these points, so that each \cpoint is isolated in the new region set.

Observe that our arcs $\gamma_i$ as well as their images are preserved by this construction, so may be viewed as lying on each of the $\pg^j$, and similarly the mapping class $\tau_\alpha^{-1} \circ \mon$ can be viewed as an element of each $\Gamma_{\pg^j}$.

 We then inductively define a sequence of right-positions $\pos^j_i$ for the pair $(\tau_\alpha^{-1} \circ \mon, \gamma_i)$, $i=1,2$, for each surface $\pg^j$, where $\pos_i^0 = \pos_i$, and $\pg^0 = \pg$, and for $j >0$, each $\pos_i^j$ is obtained from $\pos^{j-1}_i$ by removing any point which lies either in the interior of $B_j$ or in the interior of its boundary.

\begin{figure}[htb]
\centering
\includegraphics[scale=.8]{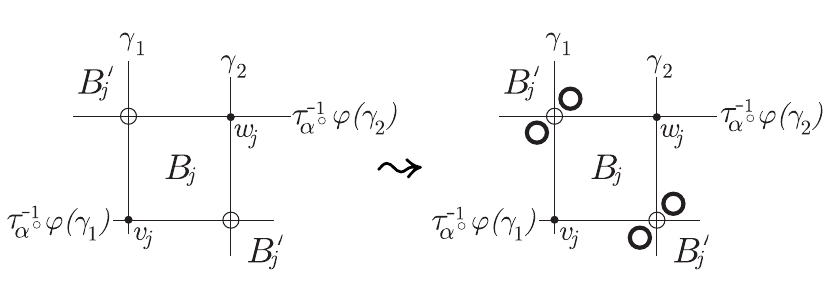}
\caption{Modification $\pg^{j-1} \leadsto \pg^j$. The thick circles indicate boundary components introduced by the modification.}
\label{fig_deletion_points}
\end{figure}

Finally:

\begin{lem}\label{lem_deletion}
 The above modification process is such that, for each $\pg^j$, the pair $\pos_1^j$ and $\pos_2^j$ is consistent, and furthermore $\mathcal{R}(\pos^*_1,\pos^*_2) \subset \mathcal{R}(\pos^j_1,\pos^j_2)$.
\end{lem}

\begin{proof}
The proof is by induction on $j$: The base step, that on $\pg$ the original positions $\pos_i$ satisfy the conditions, is trivial. Suppose then that the conditions hold for $\pg^{j-1}, \pos^{j-1}_1$, and $\pos^{j-1}_2$. Observe that, as the deletion process ensures that for each $k<j$, $B_k$ is empty, and each pair $y_{e_{2k-3}}, y_{e_{2k-2}}$ isolated, it follows from Lemma~\ref{lem_empty} that each region $B_1,B'_1,B_2, \ldots ,B_{j-1}, B'_{j-1}$ is empty in $(\pos^{j-1}_1,\pos^{j-1}_2)$. Consider then $\pg^j$, $\pos_1^j$ and $\pos_2^j$.

We first show that the pair ($\pos_1^j, \pos_2^j$) is consistent: there is an obvious inclusion of $\mathcal{R}(\pos^{j}_1,\pos^{j}_2)$ in $\mathcal{R}(\pos^{j-1}_1,\pos^{j-1}_2)$, so we must show that, if $B \in \mathcal{R}(\pos^{j-1}_1,\pos^{j-1}_2)$ does not survive the deletion process (i.e. $B$ has an \dpoint in $B_j$ or contains one of $y_{e_{2j-1}},y_{e_{2j}}$), then the same is true for any completing region $B'$. Figure~\ref{fig_deletion_regions} indicates the setup. Observe that $B$ cannot contain $B_j$; this is of course obvious for $j=1$, while for $j>1$ follows easily, as each such region either has an \dpoint in $B'_{j-1}$ or contains one of the points $\{y_{e_k} \ | \ k \leq 2j-2\}$. Suppose then that $B$ is a region with \dpoint $v$ in $B_j$ (Figure~\ref{fig_deletion_regions}(a)). Then, as $B'_{j-1}$ is empty (for $j>1$), and the pair $y_{e_{2j-3}},y_{e_{2j-2}}$ isolated, $B$ has a \cpoint in $B_j$. But then any completing region $B'$ either contains one of the points $\{y_{e_{2j-1}},y_{e_{2j}}\}$, or has an \dpoint in $B_j$. In either case $B'$ does not survive. If, on the other hand, $B$ has no \dpoint in $B_j$, then $B$ must contain one of the points $\{y_{e_{2j-1}}, y_{e_{2j}}\}$ (Figure~\ref{fig_deletion_regions}(b)), and again have a \cpoint in $B_j$, so any completing $B'$ again does not survive.

It is left then to show that no region of $\mathcal{R}(\pos^*_1,\pos^*_2)$ is removed through the deletion process. For the initial deletion, this follows by construction, as $B_1$ is of course empty and has isolated \cpoints in $(\pos^*_1,\pos^*_2)$, while the general case follows easily from Lemma~\ref{lem_structure_2}. Indeed, in the triple $(\pg^{j-1}, \pos^{j-1}_1$, $\pos^{j-1}_2)$, the region $B'_{j-1}$ is empty, and the points $\{y_{e_{2j-3}},y_{e_{2j-2}}\}$ are isolated, so by \ref{lem_structure_2} the same holds for $(B'_{j-1})^c = B_j$.
\end{proof}

\begin{figure}[htb]
\centering
\includegraphics[scale=.8]{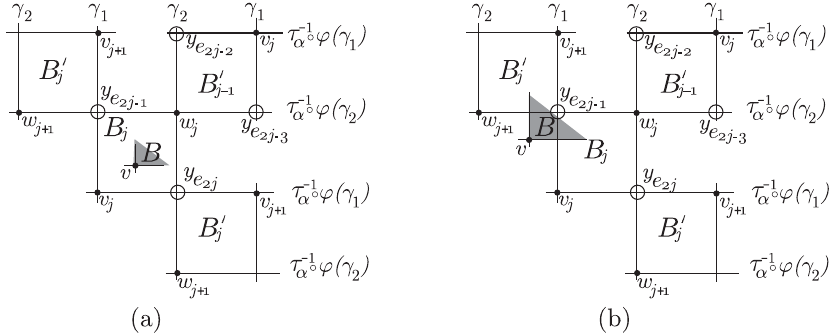}
\caption{The setup for the argument of Lemma  \ref{lem_deletion}. As usual, the figure illustrates the situation in the universal cover of the surface; lifts of a common object are labeled identically.}
\label{fig_deletion_regions}
\end{figure}

We are now able to complete Theorem~\ref{thm_nested}:

\begin{proof} (of Theorem~\ref{thm_nested}) Let $\alpha \in p.e.(\mon)$. By Corollary \ref{lem_balanced}, $\alpha$ is balanced, so  let $\pos_i$ be consistent right positions for $(\tau_\alpha^{-1} \circ  \mon, \gamma_i)$, $i=1,2$, and let $\pos^*_i \subset \pos_i$ be as in Lemma~\ref{lem_index_chain}. But by Lemma~\ref{lem_deletion}, each region $B \in \mathcal{R}(\pos^*_1, \pos^*_2)$ is empty in $\mathcal{R}(\pos^*_1,\pos^*_2)$, with isolated \cpoints , and so by Lemma~\ref{lem_structure} each $B$ is nested. Thus, using Observation~\ref{nestedness_characterization}, $\alpha$ is nested.  
\end{proof}

\section{Non-positive open books of Stein-fillable contact 3-folds}\label{sec_e}

In this section we prove Theorem \ref{thm_stein_but_not_positive} by constructing explicit examples of open book decompositions which support Stein fillable contact structures yet whose monodromies have no positive factorizations. We first introduce a construction, based on a modification of the lantern relation, which allows us to introduce essential left-twisting into a stabilization-equivalence class of open book decompositions. As, by Giroux, elements of such an equivalence class support a common contact manifold, proving the essentiality of this left-twisting (accomplished here using the methods of the previous sections) is sufficient to produce examples of non-positive open books which support Stein-fillable contact structures.

\subsection{Immersed lanterns}

For our construction, we recall the definition of a stabilization of an open book decomposition:

\begin{defin}\label{def_stabilization}
Given open book decomposition $\obd$, let $\sigma$ be a properly embedded arc in $\pg$. Then if $\pg'$ denotes the result of adding a 1-handle to $\pg$ with attaching sphere $\partial (\sigma)$, and $s$ the (unique up to isotopy) simple closed curve in $\pg'$ obtained by extending $\sigma$ over the core of this handle, then the open book decomposition $(\pg',\tau_s \circ \mon)$ is a \emph{stabilization} of $\obd$. The inverse operation is referred to as \emph{destabilization}.
\end{defin}

We start out then with a surface $\pg_{0,4}$ of genus zero with 4 boundary components, and a mapping class with factorization $\tau_{\epsilon_1} \tau_{\epsilon_2}$, where $\epsilon_1$ and $\epsilon_2$ intersect exactly twice, with opposite sign (so $\pg_{0,4}$ is just a regular neighborhood of the pair - Figure~\ref{fig_folded_lantern}(a)). It is easy to see that this defines an open book decomposition of $S^1 \times S^2$ with the standard (and unique) Stein fillable contact structure. We use the well-known lantern relation to give a mapping class equivalence between the words indicated in Figure~\ref{fig_folded_lantern}(a) and (b), thus introducing a left twist about the curve $\alpha$ into the positive monodromy. To make room for what is to come, we must enlarge the surface by adding a 1-handle (Figure~\ref{fig_folded_lantern}(c)) to the surface, so that the associated open book decomposition is now of $\#^2(S^1 \times S^2)$, with its (also unique) Stein fillable contact structure. We then stabilize, use this stabilization curve to braid two of the lantern curves into a new configuration in which the lantern relation does not apply, and then destabilize to a book in which, as the lantern is unavailable, the left twist about $\alpha$ can no longer be canceled. The steps are indicated clearly in the remainder of Figure~\ref{fig_folded_lantern}. As none of the steps $(c)$ through $(f)$ affect the supported contact structure, we have obtained an open book decomposition, supporting a Stein-fillable contact structure, which has no obvious positive factorization. We refer to this construction as an \emph{immersed} lantern relation. To motivate this terminology, observe that the surface and curves of Figure~\ref{fig_folded_lantern}(f) are obtained from those of Figure~\ref{fig_folded_lantern}(b) by a self plumbing of the surface. Figure~\ref{fig_immersed_lantern} summarizes the construction.

\begin{figure}[h!]
\centering \scalebox{.8}{\includegraphics{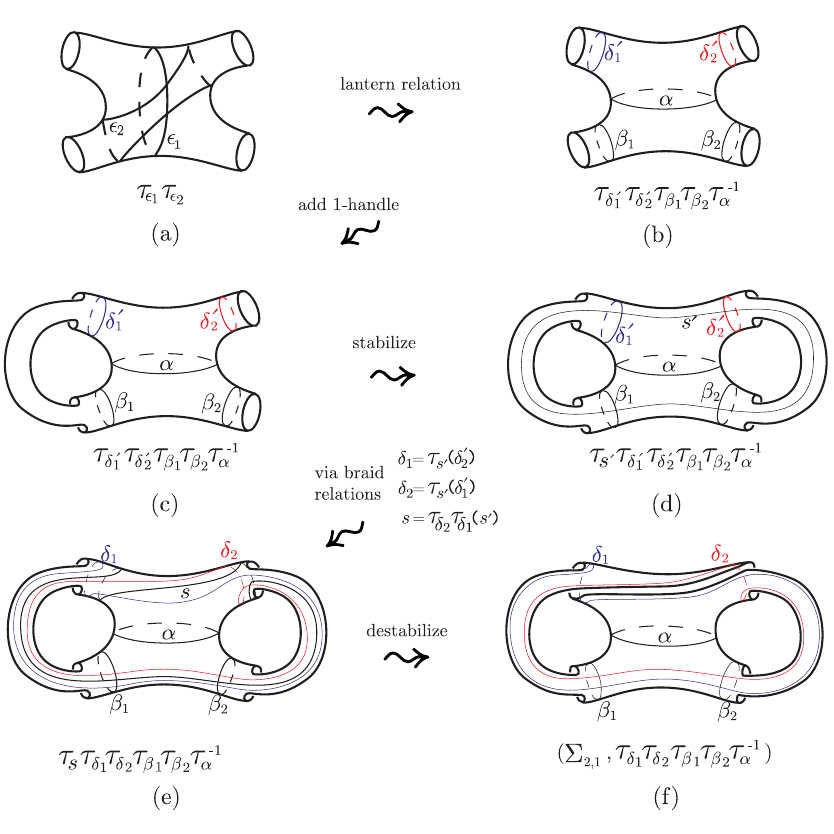}}
\caption[Construction of the counterexample]{$(\#^2 (S^1 \times S^2),\xi_{st})$ is supported by $(
\pg_{2,1}, \tau_{\delta_1} \tau_{\delta_2} \tau_{\beta_1} \tau_{\beta_2} \tau^{-1}_\alpha )$.}
\label{fig_folded_lantern}
\end{figure}

\begin{figure}[h!]
\centering \scalebox{.7}{\includegraphics{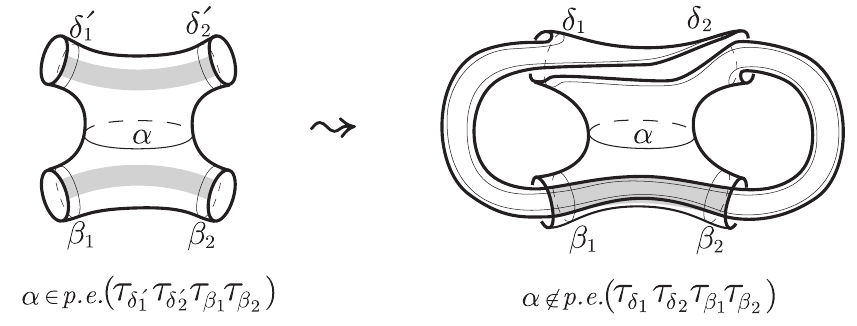}}
\caption[Immersed lantern]{Immersing a lantern: To the left, curves of the lantern relation. In particular, a positive twist about each boundary parallel curve is sufficient to cancel a negative twist about the curve $\alpha$. To the right, the curves in an `immersed' configuration for which the cancellation no longer holds. Topologically, one obtains the immersed configuration by a self-plumbing of the surface; i.e. an identification of the shaded rectangles via a $90^\circ$ twist as in the figure. As open book decompositions, one gets from one picture to the other by the steps of Figure~\ref{fig_folded_lantern}. The effect on the contact manifold is a connect sum with $S^1 \times S^2$, with the standard (Stein-fillable) contact structure.}
\label{fig_immersed_lantern}
\end{figure}

The remainder of this section uses the results of Section \ref{sec_d} to show
\begin{thm}\label{thm_example}
Let $\mon_{e_1,e_2} = \tau_{\delta_1} \tau_{\delta_2} \tau^{e_1}_{\beta_1} \tau^{e_2}_{\beta_2} \tau^{-1}_\alpha$,  where all curves, and the surface $\pg$ in which they exist, are as in Figure~\ref{fig_folded_lantern}(f). Then for each pair of positive integers $e_1$ and $e_2$, the open book decomposition $(\pg,\mon_{e_1,e_2})$ supports a Stein fillable contact structure, yet $\mon_{e_1,e_2}$ admits no factorization into positive Dehn twists.
\end{thm}

As a simple example, for the case $e_2=1$, it is simple to verify (e.g. by tracing backwards through Figure~\ref{fig_folded_lantern}) that we obtain an open book decomposition of $(S^1
\times S^2) \# L(e_1-1,e_1-2)$ for each $e_1 \geq 2$ (Figure~\ref{fig_generalization}).

\begin{figure}[htb]
\centering
\includegraphics[scale=.6]{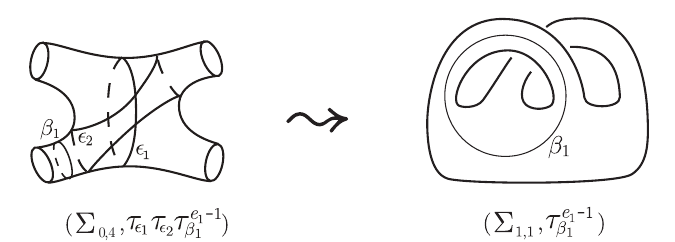}
\caption[Generalizations]{On may easily verify that the open book to the left may be destabilized twice to give the one to the right, which obviously supports $((S^1 \times S^2) \#
L(e_4-1,e_4-2),\xi_{st})$}
\label{fig_generalization}
\end{figure}

\subsection{The positive extension of $\mon_{e_1,e_2}$}

Recall (Definition \ref{def_pe}) that the \emph{positive extension}, $p.e.(\mon)$, of a positive mapping class $\mon$ is simply the set of all simple closed curves $\alpha$ such that $\tau_\alpha$ appears in a positive factorization of $\mon$.  Our method of demonstrating essentiality of the left twisting introduced by the immersed lantern is as follows: Let $\mon = \mon_{e_1,e_2}$ be as in Theorem \ref{thm_example}, for arbitrary positive integers $e_1$ and $e_2$, and define $\mon' := \tau_\alpha \circ \mon = \tau_{\delta_1}\tau_{\delta_2}\tau^{e_1}_{\beta_1}\tau^{e_2}_{\beta_2}$ (note that $\alpha$ has no intersection with any of the other curves, so $\tau_\alpha$ commutes with each twist). Suppose that
$\mon$ has positive factorization $\fac$. Then we may
factorize $\mon' = \tau_\alpha \fac$. Our goal is then to derive a
contradiction by showing that $\alpha \not\in p.e. (\mon')$. The argument, to which the remainder of this subsection is devoted, comprises two steps. Firstly, we show that $\alpha$ has trivial intersection with each curve in $p.e.(\mon')$, and secondly use this to conclude that $\alpha \not\in p.e.
(\mon')$.

For the first step, we wish to apply the results of Section
\ref{sec_d} to the pair $\gamma_1, \gamma_2$ shown in Figure~\ref{fig_examples_1} to conclude that each curve in $p.e.(\mon')$ may be isotoped so as not to intersect $\alpha$. Note that $\mon(\gamma_i)$ is isotopic to $\tau_{\delta_1}\tau_{\delta_2}(\gamma_i)$ for each $i$ - in particular, this data does not depend on $e_1$ and $e_2$. Now, $\pg \setminus \{\gamma_1,
\gamma_2\}$ is a pair of pants bounded by $\beta_1, \beta_2$, and
$\alpha$, and so these three curves are the only elements of
$SCC(\pg)$ which have no intersection with $\{\gamma_1, \gamma_2\}$.
The pair $\mon(\gamma_1)$,$\mon(\gamma_2)$ is flat (Definition \ref{def_flat}), and so by Theorem
\ref{thm_nested}, each curve $\epsilon \in p.e.(\mon') \setminus
\{\beta_1, \beta_2, \alpha\}$ is nested.

So as to simplify our view of the situation, we begin by cutting $\pg$ along $\bar{\gamma_i}, i=1,2$ (see Figure~\ref{fig_examples_1}). The
resulting surface is a pair of pants $\pg'$ with boundary components $\partial_i, i=1,2,3$, labelled such that $\alpha$ is parallel to $\partial_1$. As $\epsilon$ is balanced, we have $\# | \epsilon \cap \bar{\gamma_1} | = \#  | \epsilon \cap \bar{\gamma_2} |$.

 Examples are given in Figure~\ref{fig_examples_2}.

\begin{figure}[htb]
\centering
\includegraphics[scale=.7]{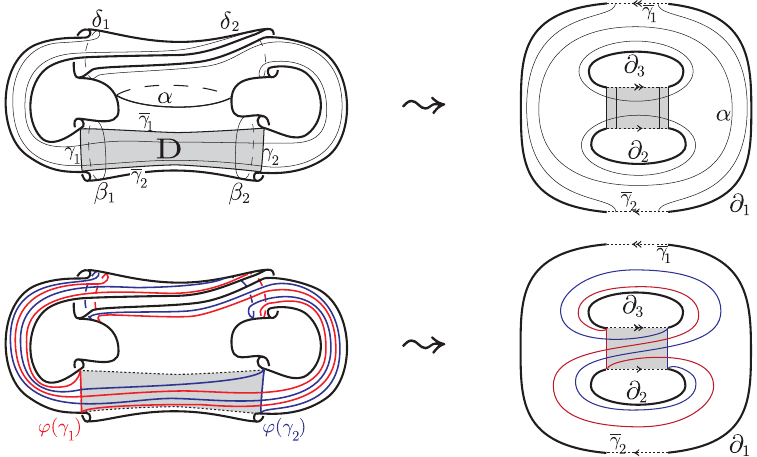}
\caption[Construction of $\pg'$]{To the left, various curves and arcs in the surface $\pg$. To the right, the result $\pg'$ of cutting $\pg$ along the $\bar{\gamma_i}$.}
\label{fig_examples_1}
\end{figure}

\begin{figure}[h!]
\centering \scalebox{.7}{\includegraphics{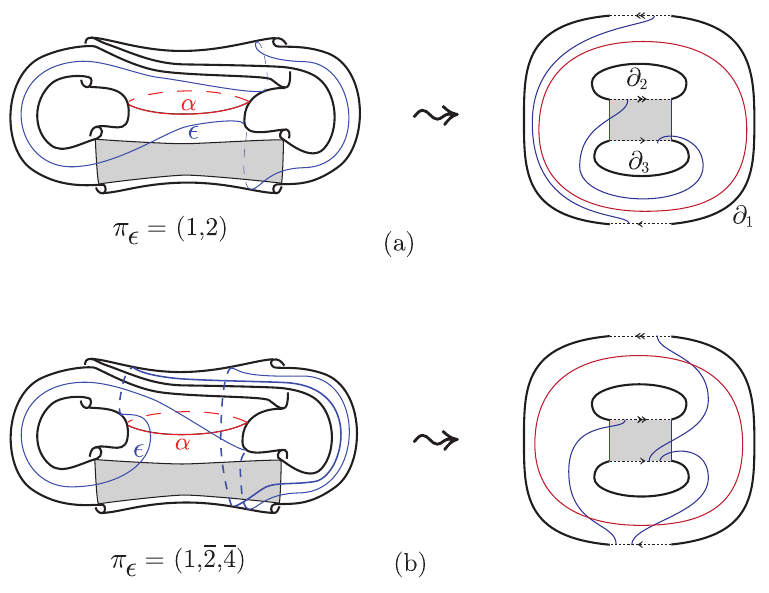}}
\caption[Curve examples on $\pg'$]{(a) A nested curve $\epsilon$ in $\pg$, and the result of cutting $\pg$ as described above. (b) A curve $\epsilon$ which is neither nested nor balanced.}
\label{fig_examples_2}
\end{figure}

We summarize various observations concerning arcs $\epsilon \cap \pg'$, for $\epsilon \in SCC(\pg)$, in the following lemma and corollary:

\begin{lem}\label{lem_arcs}
Let $\pg,\pg',\alpha$ and $D$ be as above. Then if $\epsilon$ is a nested curve on $D$,

\begin{enumerate}

\item No component of $\epsilon \cap \pg'$ has both endpoints in $\partial_2$ or both endpoints in $\partial_3$.

\item Any component of $\epsilon \cap \pg'$ with both endpoints in $\partial_1$ is boundary-parallel.

\end{enumerate}  
\end{lem} 

\begin{proof}

For (1), suppose there is some component $\gamma$ of $\epsilon \cap \pg'$ with each endpoint on $\partial_2$. Now if $\gamma$ is boundary-parallel, then it corresponds to either a bigon bounded by $\epsilon$ and one of the $\overline{\gamma_i}$, contradicting intersection minimality, or a downward arc $\epsilon \cap D$, contradicting nestedness. The same argument shows that there is no boundary parallel arc with both endpoints on $\partial_3$.  If however $\gamma$ is \emph{not} boundary parallel, it separates $\pg'$ into two annulus components, one of which contains a proper subset of $\epsilon \cap \partial_2$ on one boundary component, while the other boundary component is the original $\partial_3$. As nestedness implies that $\#|\epsilon \cap \partial_2| =  \#|\epsilon \cap \partial_3|$, this second annulus must then contain some boundary parallel arc connecting points of $\partial_3$, a contradiction.

For (2), any arc with each endpoint on $\partial_1$ which is not boundary-parallel would separate $\pg'$ into two annuli so that one annulus has the points $\epsilon \cap \partial_2$ on one boundary component, and some subset of the points $\epsilon \cap \partial_1$ on the other, while the other annulus has the points $\epsilon \cap \partial_3$ on one boundary component, again some subset of the points $\epsilon \cap \partial_1$ on the other. But, letting $m = \#|\epsilon \cap \partial_2| =   \#|\epsilon \cap \partial_3| = 1/2(\#|\epsilon \cap \partial_1|)$, there are a total of $2m-2$ remaining points of $\epsilon \cap \partial_1$. This would then necessitate some arc parallel to $\partial_2$ or $\partial_3$, contradicting (1).

\end{proof}

\begin{cor}\label{cor_arcs}
Given $\pg, \delta_1,$ and $\delta_2$ as above (Figure~\ref{fig_examples_1}), let $\psi \in Dehn^+(\pg)$ be such that $\psi(\gamma_i)$ is isotopic to $\tau_{\delta_1} \tau_{\delta_2}(\gamma_i)$, for $i=1,2$. Then for each $\epsilon \in p.e(\psi)$, any horizontal component of $\epsilon \cap D$ is contained in a component of $\epsilon \cap \pg'$ which intersects neither $\psi(\gamma_i)$ and has endpoints on $\partial_2$ and $\partial_3$. Similarly, any vertical component  of $\epsilon \cap D$ is contained in a component of $\epsilon \cap \pg'$ which intersects neither $\gamma_i$ and has endpoints on $\partial_2$ and $\partial_3$.

\end{cor}

\begin{proof}

Let $\sigma$ be a connected component of $\epsilon \cap \pg'$ which contains a horizontal segment of $\epsilon \cap D$. Now, such $\sigma$ is not parallel to $\partial_1$, so Lemma \ref{lem_arcs} tells us that the endpoints of $\sigma$ are on distinct boundary components. Assume then that $\sigma$ has an endpoint on $\partial_2$ (the case of endpoint on $\partial_3$ is identical). We consider firstly the case that the remaining endpoint is on $\partial_3$. As $\delta_1$ and $\delta_2$ are isotopic to (respectively) $\partial_2$ and $\partial_3$,  $\sigma$ is isotopic to $\tau_{\delta_1}^{t_1}\tau_{\delta_2}^{t_2}(\gamma)$, for some integers $t_1$ and $t_2$, and properly embedded arc $\gamma$ parallel to $\gamma_2$. Now, if either $t_i$ is negative, we would have a downward arc of $\epsilon \cap D$ (Figure~\ref{fig_examples_3}(a)). On the other hand, if either $t_i$ is greater than 1, we would have (for $i=2$) $\tau_\epsilon(\gamma_2) > \psi(\gamma_2)$ (from $c'_2$) (Figure~\ref{fig_examples_3}(b)). We thus have $t_i \in \{0,1\}$ for each $i$; it is immediate to check that the case $t_1=t_2=1$ is the unique case containing a horizontal arc, and further that it intersects neither $\mon(\gamma_i)$, as desired.

Suppose then that the remaining endpoint is on $\partial_1$, so that $\sigma$ is isotopic to $\tau_{\delta_1}^{t_1}\tau_{\alpha}^{t_3}(\gamma')$, where $\gamma'$ is an arc from $\partial_1$ to $\partial_3$, satisfying $\gamma' \cap \gamma_1 = \emptyset$ and $\#|\gamma' \cap \gamma_2| = 1$  (Figure~\ref{fig_examples_3}(c)). Then, for horizontality, we must have $t_1 >0$, in which case again  $\tau_\epsilon(\gamma_2) > \psi(\gamma_2)$ (from $c'_2$), giving a contradiction.

The vertical case is of course trivial, as a vertical arc of $\epsilon \cap D$ is by definition itself an arc of $\epsilon \cap \pg'$ which intersects neither $\gamma_i$ and has endpoints on $\partial_2$ and $\partial_3$.

\end{proof}

\begin{figure}[h!]
\centering \scalebox{.7}{\includegraphics{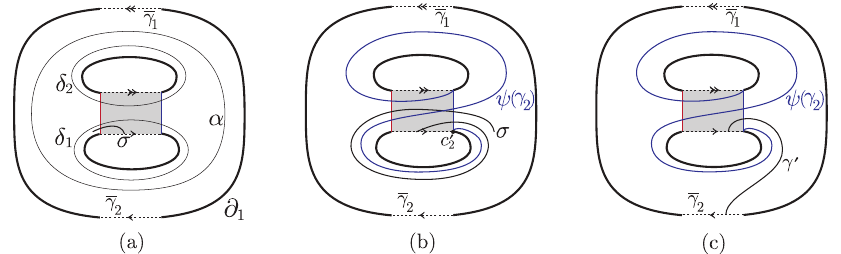}} 
\caption[]{Arcs of $\epsilon \cap \pg'$}
\label{fig_examples_3}
\end{figure}

Throughout the rest of this section, we will be interested in positive factorizations of a given monodromy only up to braid relations. As such, we write $\omega \bequiv \omega'$ to indicate that $\omega,\omega' \in Fac^+(\psi)$ are related through braid relations, and $[\omega]_b$ for the associated equivalence class. Our main lemma is the following:

\begin{lem}\label{lem_no_diagonal}
Let $\psi$ be as in the previous lemma, and $\omega$ a positive factorization of $\psi$. Then there is $\omega' \bequiv \omega$ such that no curve in $\omega'$ has any diagonal arc on $D$.
\end{lem}

\begin{proof}

Observe firstly that, after braiding, we may assume any given factorization of $\psi$ is $\tau_\alpha^e \omega_h \omega_d \omega_v$, for some $e \geq 0$, where the restriction to $D$ of each curve involved in $\omega_h$ is a collection of horizontal arcs, those in $\omega_v$ are vertical, and each element of $\omega_d$ has some diagonal arc. We will further assume our given factorization is chosen from its braid-equivalence class such that the length of $\omega_d$ is minimized, and non-zero, and derive a contradiction. Let $\omega_d = \tau_{\epsilon_n}\tau_{\epsilon_{n-1}}\cdots\tau_{\epsilon_1}$. 

Throughout the remainder of the proof, we will refer to a curve $\epsilon \in SCC(\pg)$ as \emph{vertical (horizontal)} if its restriction to $D$ contains a vertical (horizontal) arc. We begin with the case that no $\epsilon_i$ is horizontal. Recalling the ordering on arcs with common endpoint described at the beginning of this paper (Example~\ref{example_arcs}), observe that we may order SCC($\pg$) by considering the images of a given arc under twists about the curves. For our purposes, for $\epsilon,\epsilon' \in SCC(\pg)$, we set $\epsilon \geq \epsilon' \Leftrightarrow \tau_{\epsilon}(\gamma_1) \geq \tau_{\epsilon'}(\gamma_1)$ from $c_1$. 

We let $S$ denote the set of elements of $[\omega_d]_b$ which maximize the index $m := max\{i \ | \ \epsilon_i \geq \epsilon_j \forall j \leq i\}$, and then $S' \subset S$ those elements of $S$ for which the intersection number $|\epsilon_m \cap \gamma_1|$ is minimized (note that, as each curve is nested, any curve whose intersection with $D$ contains a diagonal arc intersects each of $\gamma_1$ and $\gamma_2$; in particular $|\epsilon_m \cap \gamma_1| > 0$). We may then assume $\omega_d \in S'$.

 Now, let $x$ denote the initial (from $c_1$) point of $\tau_{\epsilon_m}(\gamma_1) \cap \bar{\gamma_1}$. As $\epsilon_m$ is not horizontal, $[c_1,x]_{\tau_{\epsilon_m}(\gamma_1)}$ is contained in $D$. As $\psi(\gamma_1)$ is flat, there is some $k>m$ such that $x$ is a vertex of a downward triangular region of the triple $(\epsilon_k,\tau_{\epsilon_m}(\gamma_1),\bar{\gamma}_1)$ (Figure~\ref{fig_examples_4}(a)). After braiding, we may assume $k=m+1$. Now, as $\epsilon_k$ is nested, the segment of $\epsilon_k \cap D$ which extends from this triangular region cannot be upward on $c_2$, thus is vertical. We then have a bigon of $D_{\epsilon_k}(\tau_{\epsilon_m}(\gamma_1))$ and $\bar{\gamma_2}$; as $\tau_{\epsilon_k}\tau_{\epsilon_m} (\gamma_1)$ cannot be further to the right (from $c_1$) than $\psi(\gamma_1)$, we must have a cancelling bigon (Figure~\ref{fig_examples_4}(b)), and point $x' \in \tau_{\epsilon_m}(\gamma_1) \cap \bar{\gamma_2}$ corresponding to a vertex of this bigon (Figure~\ref{fig_examples_4}(c)). 

\begin{figure}[htb]
\centering
\includegraphics[scale=1]{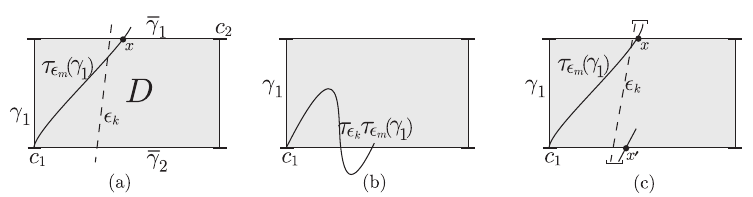}
\caption[]{}
\label{fig_examples_4}
\end{figure}

Let $A$ denote a closed support neighborhood of the twist diffeomorphism $D_{\epsilon_m}$. We label the points $\gamma_1 \cap \partial A$ as $z_0,z_1,z_2,\ldots$ increasing along $\gamma_1$ from $c_1$ (so e.g. for even $i$, $[z_i,z_{i+1}]_{\gamma_1}$ is contained in $A$) (Figure~\ref{fig_examples_A}(a)). Each $z_i$ is thus in $D_{\epsilon_m}(\gamma_1) \cap \gamma_1$; as  $D_{\epsilon_m}(\gamma_1)$ and  $\gamma_1$ bound no bigons, we may assume $\tau_{\epsilon_m}(\gamma_1)$ passes through each $z_i$, so, for even $i$, $[z_i,z_{i+1}]_{\tau_{\epsilon_m}(\gamma_1)}$ is isotopic to $\tau_{\epsilon_m}([z_i,z_{i+1}]_{\gamma_1})$. Let $r$ denote $max\{ i \ | \ z_i \in [x,x']_{\tau_{\epsilon_m}(\gamma_1)}\}$, so $\epsilon_k \cap D$ has $r+1$ connected components (Figure~\ref{fig_examples_A}(b)).

\begin{figure}[htb]
\centering
\includegraphics[scale=.9]{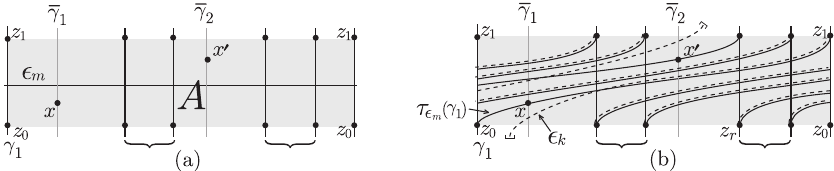}
\caption[]{The neighborhood $A$ of $\epsilon_m$, and the various arcs it intersects. Dark vertical lines are segments of $\gamma_1$.}
\label{fig_examples_A}
\end{figure}

\begin{figure}[htb]
\centering
\includegraphics[scale=.9]{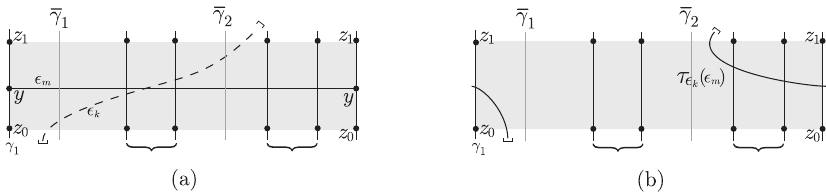}
\caption[]{}
\label{fig_examples_5}
\end{figure}

We note firstly that $r \geq 1$. Indeed, supposing otherwise, if $y$ denotes the unique point $\epsilon_m \cap [z_0,z_1]_{\gamma_1}$ , we see that $y$ is not upward in $(\epsilon_k,\gamma_1,\epsilon_m)$ (Figure~\ref{fig_examples_5}(a)), and so $\tau_{\epsilon_k}(\epsilon_m) > \epsilon_m > \epsilon_k$ (Figure~\ref{fig_examples_5}(b)). We may then braid $\tau_{\epsilon_k}\tau_{\epsilon_m} \leadsto \tau_{\tau_{\epsilon_k}(\epsilon_m)}\tau_{\epsilon_k}$, obtaining a new element of $[\omega_d]_b$ which contradicts our maximality assumption on the index $m$.

In general then, we see that $\epsilon_m$ intersects \emph{each} component of $A \cap \gamma_1$ (Figure~\ref{fig_examples_6} (a)), while $\tau^{-1}_{\epsilon_m}(\epsilon_k)$ intersects each such component at most once, and does not intersect the component $[z_r,z_{r+1}]_{\gamma_1}$ (Figure~\ref{fig_examples_6} (b)). As all intersections $\epsilon_k \cap \gamma_1$ were by construction contained in $A$, we conclude that $\#|\tau^{-1}_{\epsilon_m}(\epsilon_k) \cap \gamma_1| < \#|\epsilon_m \cap \gamma_1| <  \#|\epsilon_k \cap \gamma_1|$. But then braiding $\tau_{\epsilon_k}\tau_{\epsilon_m} \leadsto \tau_{\epsilon_m}\tau_{\tau^{-1}_{\epsilon_m}(\epsilon_k)}$, we have $\tau^{-1}_{\epsilon_m}(\epsilon_k) > \epsilon_m$, so that our new word is in $S$, and thus a contradiction of our minimality assumption on intersection with $\gamma_1$.

\begin{figure}[htb]
\centering
\includegraphics[scale=.9]{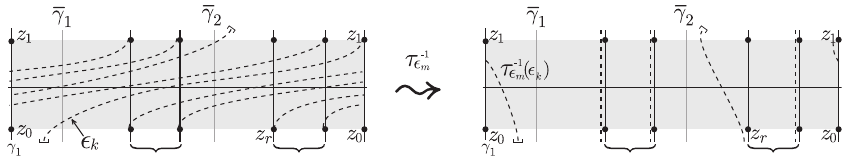}
\caption[]{}
\label{fig_examples_6}
\end{figure}

Then remaining case, that some $\epsilon_i$ is horizontal, follows from a similar argument. In particular, we now order our curves with respect to $\gamma_2$ and its endpoint $c'_2$, so $\epsilon \geq \epsilon' \Leftrightarrow \tau_{\epsilon}(\gamma_2) \geq \tau_{\epsilon'}(\gamma_2)$ from $c'_2$, and again assume $\omega_d$ is chosen from $[\omega_d]_b$ such that the index $m := max\{i \ | \ \epsilon_i \geq \epsilon_j \forall j \leq i\}$ is maximized. Now $x$ will refer to the initial point of $\tau_{\epsilon_m}(\gamma_2) \cap \bar{\gamma_1}$, and there is some $k>m$ such that $x$ is a vertex of a downward triangular region of the triple $(\epsilon_k,\epsilon_m,\bar{\gamma}_2)$. As $\epsilon_k \cap D$ contains no vertical component (else $\tau_{\epsilon_k}\epsilon_m$ is downward), the component extending from this triangle is upward on $c'_1$ (Figure~\ref{fig_examples_7}(a)). Note that, as Corollary \ref{cor_arcs} ensures that any horizontal curve is strictly greater than any non-horizontal curve, $\epsilon_m$ is horizontal, and the point of $\epsilon_m \cap \gamma_2$ closest to $c'_2$ is contained in an arc component $\sigma$ of $\epsilon_m \cap \pg'$ which connects $\partial_2$ and $\partial_3$, intersecting neither of the $\psi(\gamma_i)$.

\begin{figure}[htb]
\centering
\includegraphics[scale=.9]{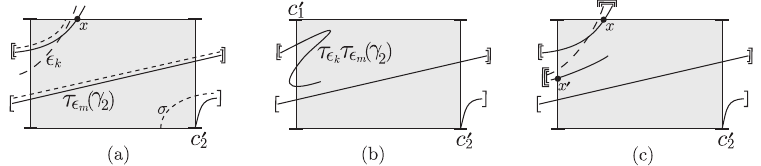}
\caption[]{The neighborhood $A$ of $\epsilon_m$, and the various arcs it intersects. Dark vertical lines are now segments of $\gamma_2$.}
\label{fig_examples_7}
\end{figure}

In particular, the segment $[c'_2,x]_{\tau_{\epsilon_m}(\gamma_2)}$ is isotopic to $[c'_2,x]_{\psi(\gamma_2)}$ up to its 2nd intersection with $\gamma_1$. Thus it again follows that we have a cancelling pair of bigons in $D_{\epsilon_k}(\tau_{\epsilon_m}(\gamma_2))$, else $\tau_{\epsilon_k}\tau_{\epsilon_m} (\gamma_2)$ is further to the right (from $c'_2$) than $\psi(\gamma_2)$ (Figure~\ref{fig_examples_7}(b)). We now let $x' \in \tau_{\epsilon_m}(\gamma_2) \cap \gamma_1$ be the point corresponding to a vertex of this bigon (Figure~\ref{fig_examples_7}(c)).

 We again consider the support neighborhood $A$ of $\tau_{\epsilon_m}$, and label $\gamma_2 \cap \partial A$ by $z_0,z_1,z_2,\ldots$, increasing along $\gamma_2$ from $c'_2$ (Figure~\ref{fig_examples_8}(a). Exactly as in the vertical case, our maximality condition on $m$ implies that $z_1 \in [x,x']_{\tau_{\epsilon_m}(\gamma_2)}$, but then $\#|\tau^{-1}_{\epsilon_m}(\epsilon_k) \cap \gamma_2| < \#|\epsilon_k \cap \gamma_2|$ (Figure~\ref{fig_examples_8}(b)), contradicting our minimality condition.

\begin{figure}[htb]
\centering
\includegraphics[scale=.9]{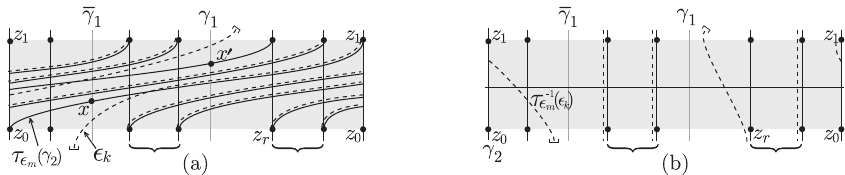}
\caption[]{}
\label{fig_examples_8}
\end{figure}

\end{proof}

We bring all of this together with:

\begin{proof}[Proof of Theorem \ref{thm_example}]

Suppose that $\mon_{e_1,e_2}$ is positive, so $\tau^{-1}_{\alpha}
\tau_{\delta_1} \tau_{\delta_2} \tau^{e_1}_{\beta_1} \tau^{e_2}_{\beta_2}$ has some positive factorization $\fac$.
Then letting $\mon' = \tau_\alpha \circ \mon_{e_1,e_2}$, we have  $\mon' = \tau_{\delta_1}\tau_{\delta_2}\tau^{e_1}_{\beta_1}\tau^{e_2}_{\beta_2} =
\tau_\alpha \fac$; i.e. $\alpha \in p.e.(\mon')$. At the same time, Lemma \ref{lem_no_diagonal} ensures that $[\tau_\alpha \fac]_b$ contains an element $\tau_{\epsilon_n} \cdots \tau_{\epsilon_1}$ with the property that for all $i$, $\epsilon_i$ is either disjoint from, or isotopic to, $\alpha$. As this property is preserved under braids, it holds for $\tau_\alpha \fac$ as well. It follows that, if $\gamma$ is a properly embedded arc which intersects $\alpha$ exactly once, then there are points $x, x' \in \mon'(\gamma) \cap \gamma$ such that $[x,x']_{\mon'(\gamma)} \cup [x,x']_{\gamma}$  is a simple closed curve isotopic to $\alpha$. The choice of $\gamma$ shown in Figure~\ref{fig_image} then gives a contradiction (we have drawn the image for the case $e_1=e_2=1$; the general case follows immediately).

\end{proof}

\begin{figure}[h!]
\centering \scalebox{.6}{\includegraphics{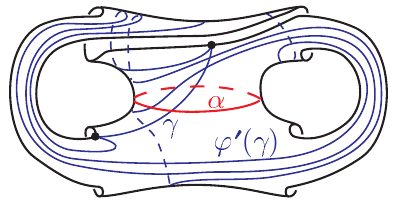}} \caption[Exclusion of $\alpha$]{The arc $\gamma$, and its image $\mon'(\gamma)$.} \label{fig_image}
\end{figure}

\end{document}